%% file: resigabem.tex
\def\revision#1{{\color{red}#1}}

\documentclass[final,5p]{elsarticle}
\usepackage{amsmath,amsthm,amssymb}
\usepackage{mathrsfs}
\usepackage{esint}  
\usepackage{color}
\usepackage{graphicx}
\usepackage{psfrag}
\usepackage{enumerate}
\usepackage{verbatim}

\def\H{\widetilde{H}}

\def\x{\boldsymbol{x}}

\def\N{\mathbb{N}}
\def\R{\mathbb{R}}
\def\Z{\mathbb{Z}}

\def\KK{\mathcal{K}}
\def\NN{\mathcal{N}}

\def\TT{\mathcal{T}}

\def\diam{{\rm diam}}

\def\supp{{\rm supp}}

\newcommand{\set}[3][\big]{#1\{#2\,:\,#3#1\}}
\newcommand{\norm}[3][]{#1\|#2#1\|_{#3}}
\def\dual#1#2{\langle#1\,;\,#2\rangle}

\def\enorm#1{|\!|\!|#1|\!|\!|}

\numberwithin{equation}{section}
\numberwithin{figure}{section}
\newtheorem{theorem}{Theorem}[section]

\newtheorem{lemma}[theorem]{Lemma}

\newtheorem{algorithm}[theorem]{Algorithm}
\newtheorem{remark}[theorem]{Remark}

\newcounter{const}

\date{\today}

\def\MM{\mathcal M}
\def\XX{\mathcal X}
\def\Crel{C_{\rm rel}}
\def\Ceff{C_{\rm eff}}

\begin{document}
\begin{frontmatter}
\title{Adaptive 2D IGA boundary element methods}
\date{\today}

\author{Michael Feischl}
\ead{Michael.Feischl@tuwien.ac.at}

\author{Gregor Gantner\corref{cor1}}
\ead{Gregor.Gantner@tuwien.ac.at}
\cortext[cor1]{Corresponding Author}

\author{Alexander Haberl}
\ead{Alexander.e101.Haberl@tuwien.ac.at}

\author{Dirk Praetorius}
\ead{Dirk.Praetorius@tuwien.ac.at}

\begin{abstract}
We derive and discuss a~posteriori error estimators for Galerkin and collocation IGA boundary element methods for weakly-singular integral equations of the first-kind in 2D. While recent own work considered the Faermann residual error estimator for Galerkin IGA boundary element methods, the present work focuses more on collocation and weighted-residual error estimators, which provide reliable upper bounds for the energy error. Our analysis allows piecewise smooth parametrizations  of the boundary, local mesh-refinement, and related standard piecewise polynomials as well as NURBS. We formulate an adaptive algorithm which steers the local mesh-refinement and the multiplicity of the knots. Numerical experiments show that the proposed adaptive strategy leads to optimal convergence, and related IGA boundary element methods are superior to standard boundary element methods with piecewise polynomials.
\end{abstract}


\end{frontmatter}

\section{Introduction} 

\input{01_introduction.tex}
\section{Preliminaries}
\label{section:preliminaries}


\input{02_preliminaries.tex}



\section{Numerical experiments}
\label{section:numerics}

\input{03_numerics.tex}


\section{A~posteriori error estimation for IGABEM}
\label{section:aposteriori}

\input{04_aposteriori.tex}



\section{Conclusion}
\label{section:conclusion}

\input{05_conclusion.tex}


\bigskip

\noindent
{\bf Acknowledgement.} The authors acknowledge support through the Austrian Science 
Fund (FWF) under grant P27005 \emph{Optimal adaptivity for BEM and FEM-BEM coupling}. 
In addition, DP and MF are supported through the FWF doctoral school \emph{Nonlinear PDEs} 
funded under grant W1245.

\bibliographystyle{alpha}
\bibliography{literature}

\end{document}

%% file: 01_introduction.tex

\subsection{Isogeometric analysis}
The central idea of isogeometric analysis (IGA) is to use the same ansatz functions for the discretization of the partial differential equation at hand as for the representation of the problem geometry. Usually,  $\Omega$ is represented in computer aided design (CAD) by means of NURBS, hierarchical splines, or T-splines. This concept, invented in~\cite{hughes2005} for finite element methods (IGAFEM) has proved very fruitful in applications~\cite{hughes2005,simpson2012}; see also the monograph \cite{bible}. Since CAD directly provides a parametrization of the boundary $\partial \Omega$, this makes the boundary element method (BEM) the most attractive numerical scheme, if applicable (i.e., provided that the fundamental solution of the differential operator is explicitly known). Isogeometric BEM (IGABEM) has first been considered in~\cite{igabem2d} for 2D resp. \cite{igabem3d} for 3D. While standard BEM with piecewise polynomials is well-studied in the literature, cf.~the monographs~\cite{ss,steinbach} and the references therein, the numerical analysis of IGABEM in essentially open. We refer to~\cite{simpson2012,laplaceiga,helmholtziga} for numerical experiments and to~\cite{stokesiga} for some quadrature analysis. A~posteriori error estimation has first been considered for Galerkin IGABEM in our recent work~\cite{fgp14}. In the present work, we extend the latter result to collocation IGABEM which is preferred in practice for its simpler assembly of the stiffness matrix.
\bigskip

\subsection{Model problem}
Let $\Omega\subset\R^2$ be a Lipschitz domain and $\Gamma\subseteq \partial\Omega$ be a compact, piecewise smooth part of the boundary with finitely many connected components. 
For a given right-hand side $f$, we consider the weakly-singular boundary integral equation
\begin{align}\label{eq:strong}
 V\phi(x):=-\frac{1}{2\pi}\int_\Gamma\log|x-y|\,\phi(y)\,dy = f(x)
 \quad\text{on }\Gamma
\end{align}
associated with the 2D Laplacian; see Section~\ref{section:preliminaries} below for the mathematical setting and the definition of the problem related energy norm $\enorm{\cdot}$. With some discrete ansatz space $\XX_h\subset L^2(\Gamma)$, the Galerkin BEM computes the unique solution $\phi_h\in\XX_h$ of the discrete variational formulation
\begin{align}\label{eq:discrete}
 \int_{\Gamma_{}} V\phi_h\,\psi_h\,dx
 = \int_{\Gamma_{}} f\psi_h\,dx
 \quad\text{for all }\psi_h\in\XX_h.
\end{align}
Note that $\XX_h\subset {L^{2}(\Gamma_{})}$ ensures $V\phi_h\in C(\Gamma)$. 
The collocation BEM computes $\phi_h\in\XX_h$ such that
\begin{align}\label{eq:collocation}
 V\phi_h(x_j) = f(x_j)
 \quad\text{for all }x_j\in \{x_1,\dots,x_{N_{\rm{col}}}\},
\end{align}
where the $x_j$ are appropriately chosen collocation points with $N_{\rm{col}}:=\dim\XX_h$; see Section \ref{subsec:collocation igabem}.
In either case~\eqref{eq:discrete}--\eqref{eq:collocation}, $\phi_h$ is computed by solving a linear system of equations

\subsection{A~posteriori error estimation for Galerkin IGABEM}
We assume that $\XX_h$ is associated to some partition $\TT_h$ of $\Gamma$ into a set of connected segments. For each
vertex $z$ of $\TT_h$, let $\omega_h({z}) := \bigcup\set{T\in\TT_h}{{z}\in T}$ 
denote the node patch. If $\XX_h$ is sufficiently rich (e.g., $\XX_h$ contains 
certain splines or NURBS), it is proved in~\cite{fgp14}
that Galerkin BEM guarantees reliability and efficiency
\begin{subequations}\label{eq:faermann}
\begin{align}
\begin{split}
 \Crel^{-1}\,\enorm{\phi-\phi_h}
 &\le \eta_h:=\Big(\sum_{{z}\in\NN_h}\eta_h(z)^2\Big)^{1/2}
\\& \le \Ceff\,\enorm{\phi-\phi_h}
\end{split}
\end{align}
with $\XX_h$-independent constants $\Ceff,\Crel>0$. Here, $r_h:=f-V\phi_h$ denotes the residual and
\begin{align}
 \eta_h(z)^2
 := \int_{\omega_h({z})}\!\int_{\omega_h({z})}\hspace*{-4mm}\frac{|r_h(x)-r_h(y)|^2}{|x-y|^2}\,dy\,dx
\end{align}
\end{subequations}
is some Sobolev-Slobodeckij seminorm, i.e., the unknown BEM energy
error  is controlled by some computable a~posteriori error estimator
$\eta_h$. Estimate~\eqref{eq:faermann} has first been
proved by Faermann~\cite{faermann2d} for closed $\Gamma_{}=\partial\Omega$ and standard spline spaces $\XX_h$ based on the
arclength parametrization.
Her result is generalized in
\cite{fgp14} to a more general setting which also includes isogeometric analysis. We note that~\cite{faermann2d,fgp14} show that the efficiency estimate $\eta_h\le C_{\rm eff}\,\enorm{\phi-\phi_h}$ holds even independently of the discretization and, in particular, for collocation.

\subsection{A~posteriori error estimation for collocation IGABEM}
In the present manuscript, we focus on the weighted-residual error estimator which has first been proposed in \cite{cs96,cc97} for standard BEM with piecewise polynomials and polygonal $\Gamma$. 
We prove that for Galerkin IGABEM~\eqref{eq:discrete} as well as collocation IGABEM~\eqref{eq:collocation}, 
there holds the upper bound
\begin{subequations}\label{eq:residual}
\begin{align}
 \begin{split}
 &\Crel^{-1}\,\enorm{\phi-\phi_h}\le \mu_h
 :=\Big(\sum_{{z}\in\NN_h}\mu_h(z)^2\Big)^{1/2},
 \end{split}
\end{align}
with an $\XX_h$-independent constant $\Crel>0$. Here, $r_h:=f-V\phi_h$ is again the residual and 
\begin{align}
 \mu_h(z)^2:={|\omega_h(z)|}\int_{\omega_h(z)} |r_h'(x)|^2\,dx
\end{align}
\end{subequations}
is a weighted $H^1$-seminorm,
where $(\cdot)'$ denotes the arc-length derivative and $|\omega_h(z)|$ is the length of the node patch.
For collocation BEM, we thus control the energy error by
\begin{align}
\begin{split}
 \Ceff^{-1}\,\eta_h 
 \le \enorm{\phi-\phi_h}
 \le\Crel\,\mu_h
\end{split}
\end{align}
which, however, involves different error estimators.
In addition to the global relation of the error estimators $\eta_h$ and $\mu_h$ and independently of the discretization, 
we prove
\begin{align}\label{eq:bound}
 \eta_h(z) \le C_{\rm loc}\,\mu_h(z)
 \quad\text{for all vertices $z$ of $\TT_h$},
\end{align}
where $C_{\rm loc}>0$ depends only on $\Gamma$.

\subsection{Outline}
Section~\ref{section:preliminaries} recalls the functional analytic framework,
provides the assumptions on $\Gamma$ and its parametrization $\gamma$, introduces the ansatz spaces, and presents an adaptive algorithm which is capable to control and adapt the multiplicity of the nodes as well as the local mesh-size {(Algorithm \ref{the algorithm})}. 
Section~\ref{section:numerics} provides the numerical evidence that the proposed adaptive IGABEM is superior to IGABEM with uniform mesh-refinement as well as to adaptive standard BEM with piecewise polynomials. Moreover,
we observe that collocation IGABEM leads to essentially the same convergence behavior
as Galerkin IGABEM, so that an adaptive collocation IGABEM may be favorable in practice.
Section~\ref{section:aposteriori} recalls the precise 
statement of~\eqref{eq:faermann} from \cite{fgp14} and gives a proof of~\eqref{eq:residual}--\eqref{eq:bound}. 
The concluding Section~\ref{section:conclusion} comments on our overall findings, open questions, and future research.

%% file: 02_preliminaries.tex
\noindent
In this section, we collect the main assumptions on the boundary and its discretization and introduce the BEM ansatz spaces. Further details on Sobolev spaces and the functional analytic setting of weakly-singular integral equations, are found, e.g., in the monographs \cite{hsiao,mclean,ss} and the references therein.

Throughout, $|\cdot|$ denotes the absolute value of scalars, the Euclidean norm of vectors in $\R^2$, the measure of a set in $\R$ (e.g., the length of an interval), or the arclength of a curve in $\R^{2}$.
The respective meaning will be clear from the context.

We write $A\lesssim B$ to abbreviate $A\le cB$ with some constant $c>0$ which is clear from the context.
Moreover $A\simeq B$ abbreviates $A\lesssim B\lesssim A$.

\def\Cgamma{C_\Gamma}
\subsection{Function spaces}
For any measurable subset $\omega\subseteq\Gamma$ resp.\ any interval $\omega\subseteq \R$,  $L^2(\omega)$ denotes
the Lebesgue space of all square integrable functions with corresponding norm
\begin{align}
 \norm{u}{L^2(\omega)}^2:=\int_\omega |u(x)|^2\,dx.
\end{align}
If $u\in L^2(\omega)$ is differentiable along the arc, $u'$ denotes the arclength derivative.
Define the Sobolev space $H^1(\omega) := \set{u\in L^2(\omega)}{u'\in L^2(\omega)}$ with corresponding norm
\begin{subequations}
\begin{align}
 \norm{u}{H^1(\omega)}^2
 &:= \norm{u}{L^2(\omega)}^2
 + |u|_{H^1(\omega)}^2,\\
 |u|_{H^1(\omega)}^2 &:= \int_\omega |u'(x)|^2\,dx.
\end{align}
\end{subequations}
Furthermore, define the Sobolev space
$H^{1/2}(\omega) := \set{u\in L^2(\omega)}{\norm{u}{H^{1/2}(\omega)}<\infty}$  with corresponding
norm
\begin{subequations}
\begin{align}\label{eq:SS-norm}
 \norm{u}{H^{1/2}(\omega)}^2
 &:= \norm{u}{L^2(\omega)}^2
 + |u|_{H^{1/2}(\omega)}^2,\\
 |u|_{H^{1/2}(\omega)}^2 &:= \int_\omega\int_\omega\frac{|u(x)-u(y)|^2}{|x-y|^2}\,dy\,dx.
\end{align}
\end{subequations}
The dual space of $H^{1/2}(\omega)$ is $\H^{-1/2}(\omega)$, where duality
is understood with respect to the extended $L^2(\omega)$-scalar product, i.e.,
for $u\in H^{1/2}(\omega)$ and $\phi\in L^2(\omega)$, it holds
\begin{align}
 \dual{u}{\phi}_\omega = \int_\omega u(x)\phi(x)\,dx.
\end{align}
We note that $H^{1/2}(\Gamma)\subset L^2(\Gamma)\subset \H^{-1/2}(\Gamma)$ form a Gelfand triple and all inclusions are dense and compact.

Amongst other equivalent definitions of $H^{1/2}(\omega)$ are the characterization
as trace space of functions {in} $H^1(\Omega)$ as well as equivalent interpolation techniques.
All these definitions provide the same space but different norms, where
norm equivalence constants depend only  on $\omega$.

\subsection{Weakly-singular integral equation}
The operator $V$ from~\eqref{eq:strong} extends to a linear and continuous operator $V:\H^{-1/2}(\Gamma)\to H^{1/2}(\Gamma)$ with additional stability $V:L^2(\Gamma)\to H^1(\Gamma)$.
We additionally suppose that $V$ is even an elliptic isomorphism, which is satisfied, e.g., if $\diam(\Omega)<1$. In particular, $\dual{V(\cdot)}{(\cdot)}_\Gamma$ is thus a scalar product on $\H^{-1/2}(\Gamma)$, and the induced energy norm
\begin{align}
 \enorm{\psi}^2 := \dual{V\psi}{\psi}_\Gamma
 \quad\text{for }\psi\in\H^{-1/2}(\Gamma)
\end{align}
is an equivalent norm on $\H^{-1/2}(\Gamma)$.

Given $f\in H^{1/2}(\Gamma)$, the weakly-singular integral equation~\eqref{eq:strong} is equivalently reformulated in variational form: Find $\phi\in\H^{-1/2}(\Gamma)$ such that
\begin{align}\label{eq:weak}
 \dual{V\phi}{\psi}_\Gamma
 = \dual{f}{\psi}_\Gamma \quad\text{for all }\psi\in\H^{-1/2}(\Gamma_{}).
\end{align}
The Lax-Milgram lemma thus applies and proves existence and uniqueness of the solution $\phi\in\H^{-1/2}(\Gamma)$ of~\eqref{eq:weak} resp.~\eqref{eq:strong}. 

\subsection{Parametrization of boundary}
\label{subsec:boundary parametrization}
Let $\Gamma = \bigcup_i\Gamma_i\subseteq \partial\Omega$ be decomposed into its finitely many connected
components $\Gamma_i$. Then, 
\begin{align*}
 &\|u\|^2_{H^{1/2}(\Gamma)}\simeq \sum_{i} \|u\|^2_{H^{1/2}(\Gamma_i)}\quad \text{for all } u\in H^{1/2}(\Gamma);
\end{align*}
see, e.g., \cite[Section 2.2]{fgp14}.
The usual piecewise polynomial and NURBS basis functions have connected support
and are hence supported by some \emph{single} $\Gamma_i$ each. Without loss of generality and to ease the mathematical proofs, we may therefore  assume that $\Gamma$ is connected. 
All results remain valid for non-connected $\Gamma$.

We assume that either $\Gamma=\partial\Omega$ is parametrized by a closed continuous and
piecewise {two times} continuously differentiable path $\gamma:[a,b]\to\Gamma$ such
that the restriction $\gamma|_{[a,b)}$ is even bijective, or that $\Gamma\subsetneqq\partial\Omega$ is parametrized by a bijective continuous and piecewise two times continuously differentiable path $\gamma:[a,b]\to\Gamma$.  
For $\Gamma=\partial\Omega$, we denote the $(b-a)$-periodic extension to $\R$ also by $\gamma$.
For the left and right derivative of $\gamma$, we assume that $\gamma^{\prime_\ell}(t)\neq 0$ for $t\in(a,b]$ and $\gamma^{\prime_r}(t)\neq 0$  for $t\in [a,b)$.
Moreover we assume that $\gamma^{\prime_\ell}(t)
+c\gamma^{\prime_r}(t)\neq0$ for all $c>0$ and $t\in[a,b]$ resp. $t\in(a,b)$.

By  $\gamma_L:[0,L]\to\Gamma$, we denote the arclength parametri\-zation, i.e.,
$|\gamma_L^{\prime_\ell}(t)| = 1 = |\gamma_L^{\prime_r}(t)|$, and its periodic extension. Then, elementary differential geometry yields bi-Lipschitz continuity
\begin{align}\label{eq:bi-Lipschitz}
\begin{split}
 \Cgamma^{-1} \le \frac{|\gamma_L(s)-\gamma_L(t)|}{|s-t|}\le\Cgamma
 \quad\text{for all }s,t\in\R, \\{\text{ with }\begin{cases}
 |s-t|\le \frac{3}{4}\,L, \text{ for }\Gamma=\partial\Omega,\\
  s\neq t\in [0,L], \text{ for }\Gamma\subsetneqq\partial\Omega;
\end{cases}}
\end{split}
\end{align}
see, e.g., \cite[Lemma 2.1]{diplarbeit} for the proof for $\Gamma=\partial\Omega$ which even simplifies for $\Gamma\subsetneqq \partial\Omega$.
%
Let $I\subseteq [a,b]$.
Suppose $|I|\le \frac{3}{4}L$ for $\Gamma=\partial\Omega$.
Then,  \eqref{eq:bi-Lipschitz} implies
\begin{align}\label{eq:equivalent Hsnorm}
\begin{split}
\Cgamma^{-1}|u\circ\gamma_{L}|_{H^{1/2}(I)} &\leq |u|_{H^{1/2}(\gamma_L(I))}
\\&
\leq \Cgamma|u\circ\gamma_{L}|_{H^{1/2}(I)}
\end{split}
\end{align}
for all $u\in H^{1/2}(\Gamma)$.
\subsection{Discretization of boundary}

For the discretization, let $\mathcal{T}_h=\{T_1,\dots,T_n\}$ be a partition of $\Gamma$ into compact and connected segments $T_j$.
The endpoints of the elements of $\mathcal{T}_h$ form the set of nodes 
\begin{align*}
\mathcal{N}_h=\begin{cases}
\set{z_j}{j=1,\dots,n}&\text{ for }\Gamma=\partial\Omega,\\
\set{z_j}{j=0,\dots,n}&\text{ for }\Gamma\subsetneqq\partial\Omega.
\end{cases}
\end{align*}
The arclength of each element $T\in \mathcal{T}_h$ is denoted by $h_{T}$.
Moreover,  the shape regularity constant is defined by
\begin{align*}
\kappa(\mathcal{T}_h)&:=\max\set{h_{T}/h_{T'}}{T,T'\in\mathcal{T}_h, T\cap T'\neq \emptyset}.
\end{align*}
For $\Gamma=\partial\Omega$,  we extend the nodes, elements, and their length periodically. Moreover, we suppose
\begin{align}\label{eq:h small}
\max_{T\in\TT_h} h_T  \le |\Gamma|/4
\quad\text{for }\Gamma=\partial\Omega.
\end{align}

\subsection{Discretization of parameter domain}

Given $\gamma:[a,b]\to \Gamma$, the partition $\mathcal{T}_h$ induces a partition $\check{\mathcal{T}}_h=\{\check{T}_1,\dots,\check{T}_n\}$ of the parameter domain $[a,b]$.
Let $a=\check{z}_0<\check{z}_1<\dots<\check{z}_n=b$ be the endpoints of the elements of $\check{\mathcal{T}}_h$. 
We assume $\check{T}_j=[\check{z}_{j-1},\check{z}_j]$, $\gamma(\check{T_j})=T_j$, and $\gamma(\check{z}_j)=z_j$. 
We define 
\begin{align*}
\check{\mathcal{N}}_h=\begin{cases}
\set{\check{z}_j}{j=1,\dots,n}&\text{ for }\Gamma=\partial\Omega,\\
\set{\check{z}_j}{j=0,\dots,n}&\text{ for }\Gamma\subsetneqq \partial\Omega.
\end{cases}
\end{align*}
The length of each $\check{T}\in \check{\mathcal{T}}_h$ is denoted by $h_{\check{T}}$.
Moreover, we define the shape regularity constant on $[a,b]$ by
\begin{align*}
\kappa(\check{\mathcal{T}_h})&:=\max\set{h_{\check{T}}/h_{\check{T}'}}{\check{T},\check{T}'\in\check{\mathcal{T}_h}, \gamma(\check{T})\cap \gamma(\check{T}')\neq \emptyset}.
\end{align*}
Note that $\kappa({\mathcal{T}_h})\simeq\kappa(\check{\mathcal{T}_h})$, where the hidden constants depend only on the parametrization $\gamma$.

\subsection{B-splines and NURBS in the parameter domain}
We consider \textit{knots} $\check{\mathcal{K}}:=(t_i)_{i\in\Z}$ on $\R$ with $t_{i-1}\leq t_{i}$ for $i\in \Z$ and $\lim_{i\to \pm\infty}t_i=\pm \infty$.
For the multiplicity of any knot $t_i$, we write $\#t_i$.
We denote the corresponding set of \textit{nodes} $\check{\mathcal{N}}:=\set{t_i}{i\in\Z}=\set{\check{{z}}_{j}}{j\in \Z}$ with $\check{{z}}_{j-1}<\check{{z}}_{j}$ for all $j\in\Z$. 
For $i\in\Z$ and $p\in\N_0$, the $i$-th \textit{B-Spline} of degree $p$ is defined inductively by
\begin{align*}
B_{i,p}^{\check{\mathcal{K}}}
\!:=\!B_{i,p}
\!:=\!
\begin{cases}
\chi_{[t_{i-1},t_{i})}&\!\!\!\text{for }p=0,\\
\beta_{i-1,p} B_{i,p-1}+(1\!-\!\beta_{i,p}) B_{i+1,p-1}&\!\!\!\text{for } p>0,
\end{cases}
\end{align*}
where, for $t\in \R$,
\begin{align*}
\beta_{i,p}(t):=
\begin{cases}
\frac{t-t_i}{t_{i+p}-t_i} \quad &\text{for } t_i\neq t_{i+p},\\
0 \quad &\text{for } t_i= t_{i+p}.
\end{cases}
\end{align*}
We collect some basic properties of B-splines from \cite{Boor-SplineBasics}:

\begin{lemma}[{\cite[Theorem 6, Section 2 and page 9--10]{Boor-SplineBasics}}]
\label{lem:properties for B-splines}
 For $p\in \N_0$, the following assertions hold: 
\begin{enumerate}[\rm(i)]
\item \label{item:spline basis} Let $I=[a,b)$ be a finite interval.
Then,
\begin{equation}
\set{B_{i,p}|_I}{i\in \Z, B_{i,p}|_I\neq 0}
\end{equation}
is a basis for the space of all right-continuous $\check{\mathcal{N}}$-piece\-wise polynomials of degree lower or equal $p$ on $I$ and which are, at each knot $t_i$, $p-\#t_i$ times continuously differentiable if $p-\#t_i\geq 0$.

\item \label{item:B-splines local} For $i\in\Z$,  $B_{i,p}$ vanishes outside the interval $[t_{i-1},t_{i+p})$. 
It is positive on the open interval $(t_{i-1},t_{i+p})$.
\item \label{item:B-splines determined} For $i\in \Z$,  $B_{i,p}$ is completely determined by the $p+2$ knots $t_{i-1},\dots,t_{i+p}$. 
\item\label{item:B-splines partition} The  B-splines of degree $p$ form a locally finite partition of unity, i.e.,
$\sum_{i \in\Z} B_{i,p}=1\text{ on }\R$.\qed
\end{enumerate}
\end{lemma}

In addition to the knots $\check{\mathcal{K}}=(t_i)_{i\in\Z}$, we consider weights $\mathcal{W}:=(w_i)_{i\in\Z}$ with $w_i>0$.
For $i\in \Z$ and $p\in \N_0$, we define the $i$-th  \textit{non-uniform rational B-Spline} (\textit{NURBS}) of degree $p$
\begin{equation}
R_{i,p}^{\check{\mathcal{K}},\mathcal{W}}:=R_{i,p}:=\frac{w_iB_{i,p}}{\sum_{\ell\in\Z}  w_{\ell}B_{\ell,p}}.
\end{equation}
Note that the denominator is positive and locally finite.

For any $p\in\N_0$, we define the vector spaces
\begin{align}
\mathscr{S}^p(\check{\mathcal{K}})&:=\left\{\sum_{i\in\Z}a_i B_{i,p}^{\check{\mathcal{K}}}:a_i\in\R\right\},\\
\label{eq:NURBS space defined}
\mathscr{N}^p(\check{\mathcal{K}},\mathcal{W})&:=\left\{\sum_{i\in\Z}a_i R_{i,p}^{\check{\mathcal{K}},\mathcal{W}}:a_i\in\R\right\}.
\end{align}

\subsection{NURBS on the boundary}

For $\Gamma=\partial\Omega$, each node $\check{z}\in \check{\mathcal{N}}_h$ has a multiplicity $\#\check{z}\leq p+1$. 
This induces a sequence of non-decreasing knots $\check{\mathcal{K}}_h=(t_i)_{i=1}^N$ on $(a,b]$.
Let $\mathcal{W}_h=(w_{i})_{i=1}^N$ be a sequence of weights on these knots. 
We extend the knot sequence $(b-a)$-periodically to $(t_i)_{i\in \Z}$ and the weight sequence to $(w_i)_{i\in \Z}$ by $w_{N+i}:=w_i$ for $i\in \Z$.
For the extended sequences, we also write $\check{\mathcal{K}}_h$ and $\mathcal{W}_h$. 
We set
\begin{equation}
\widehat{\mathscr{N}}^p(\check{\mathcal{K}}_h,\mathcal{W}_h):=\mathscr{N}^p(\check{\mathcal{K}}_{h},\mathcal{W}_{h})|_{[a,b)}\circ \gamma|_{[a,b)}^{-1}.
\end{equation}

For $\Gamma_{}\subsetneqq \partial\Omega$, each node $\check{z}\in \check{\mathcal{N}}_h$ has a  multiplicity $\#\check{z}\leq p+1$ such that $\#\check{z}_0=\#\check{z}_n=p+1$.
This induces a sequence of non-decreasing knots {\color{black}$\check{\mathcal{K}}_h=(t_i)_{i=-p}^N$} on $[a,b]$.
Let {\color{black}$\mathcal{W}_h=(w_{i})_{i=1}^{N}$} be a sequence of weights.
We extend the sequences arbitrarily to $\check{\mathcal{K}}_h=(t_i)_{i\in\Z}$ with $t_i\leq t_{i+1}$ for $i\in\Z$, $a>t_{i}\to-\infty$ for $i<-p$, and $b<t_i\to \infty$ for $i>N$, and $\mathcal{W}_h=(w_i)_{i\in \Z}$ with $w_i>0$. 
We set
\begin{equation}
\widehat{\mathscr{N}}^p(\check{\mathcal{K}}_h,\mathcal{W}_h):=\mathscr{N}^p(\check{\mathcal{K}}_h,\mathcal{W}_h)|_{[a,b]}\circ \gamma^{-1}.
\end{equation}
Due to Lemma \ref{lem:properties for B-splines} \eqref{item:B-splines local}--\eqref{item:B-splines determined}, this definition does not depend on how the sequences are extended.

\subsection{Collocation IGABEM}\label{subsec:collocation igabem}

In this section, we show how to choose the collocation points $x_j$ for $j=1,\dots,N_{\rm col}$ in~\eqref{eq:collocation}.
First, we note that Lemma~\ref{lem:properties for B-splines} \eqref{item:spline basis} implies that
\begin{align}\label{eq:closed basis}
\big\{ R_{i,p}|_{[a,b)}:i=1-p,\dots,N-\#b+1\big\}\circ\gamma|_{[a,b)}^{-1}
\end{align}
for $\Gamma=\partial\Omega$ resp.
\begin{align}\label{eq:open basis}
\big\{ R_{i,p}|_{[a,b]}:i=1-p,\dots,N-p\big\}\circ\gamma^{-1}
\end{align}
for $\Gamma\subsetneqq\partial\Omega$
forms a basis of $\widehat{\mathscr{N}}^p(\check{\KK}_h,\mathcal{W}_h)$.
Recall $\#b=p+1$ for $\Gamma\subsetneqq\partial\Omega$.
For simplicity, suppose $\#b=p+1$ also for $\Gamma=\partial\Omega$.
This gives
\begin{align}
N_{\rm col}=N.
\end{align}
For $j=1,\dots,N$, the collocation point $x_j$ is defined through the arithmetic mean of $p+2$ knots in the parameter domain
\begin{align}
 x_j = \gamma(\check{x}_j)
 \quad\text{with}\quad
 \check x_j:=\frac{\sum_{k=j-p-1}^{j} t_k}{p+2}.
\end{align}

\subsection{Adaptive algorithm}\label{subsec:algorithm}

Finally, we recall an adaptive algorithm from our preceding work~\cite{fgp14}, which 
steers the $h$-refinement of the partition $\mathcal{T}_h$ as well as the increase of the multiplicity of the nodes $\mathcal{N}_h$. 
While \cite{fgp14} considered $\eta_h$ for Galerkin IGABEM, the current focus is on $\mu_h$ and collocation IGABEM.

Suppose that $\Gamma$ is represented by a NURBS curve of degree $p\in\N_0$.
This induces the initial partition $\TT_0$ of $\Gamma$ with nodes $\mathcal N_0$, 
related nodes $\check{\mathcal{N}}_0$ in the parameter domain, and 
positive weights $\mathcal{W}_0$. Each node has a multiplicity lower or equal 
$p+1$, where for $\Gamma\subsetneqq \partial\Omega$ or collocation IGABEM we 
suppose $\#a=\#b=p+1$. For $\Gamma=\partial\Omega$, we suppose $h_T\leq |\Gamma|/4$ 
for all $T\in \mathcal{T}_0$.

As the initial trial space, we consider
\begin{equation}
\widehat{\mathscr{N}}^p(\check{\mathcal{K}}_0,\mathcal{W}_0)\subset L^2(\Gamma_{})\subset{H}^{-1/2}(\Gamma_{}).
\end{equation}
Fix an error estimator $\varrho_h\in\{\eta_h,\mu_h\}$.
The nodal contributions $\varrho_h(z)$ from \eqref{eq:faermann} resp.~\eqref{eq:residual} are used to steer knot insertion from $\check{\mathcal{K}}_{h}$ to the following knots $\check{\mathcal{K}}_{H}$.
The new weights $\mathcal{W}_{H}$ are uniquely chosen such that the denominator of the NURBS functions does not change. In particular, this implies nestedness 
\begin{equation}
\widehat{\mathscr{N}}^p(\check{\mathcal{K}}_h,\mathcal{W}_h)\subseteq \widehat{\mathscr{N}}^p(\check{\mathcal{K}}_H,\mathcal{W}_H)
\end{equation}
 of the related NURBS spaces.
Since the weights in $\mathcal W_H$ are just convex combinations of the weights in $\mathcal{W}_0$, it holds $\min \mathcal{W}_0 \le \min\mathcal{W}_{H}\le \max\mathcal{W}_{H}\le \max\mathcal{W}_0$.
For details, we refer to \cite[Section 4.2]{fgp14}.

Then, the adaptive algorithm reads as follows:

\begin{algorithm}\label{the algorithm}
\textbf{Input:} Adaptivity parameter $0<\theta\le1$, polynomial order $p\in \N_0$, initial partition $\TT_0=\mathcal{T}_h$ with knots $\check{\KK}_0=\check{\KK}_h$, initial weights $\mathcal{W}_0=\mathcal{W}_h$.\\
\textbf{Adaptive loop:} Iterate the following steps {\rm(i)}--{\rm(vi)}, until $\varrho_h$ is sufficiently small:
\begin{itemize}
\item[\rm(i)] Compute  approximation $\phi_h\in \widehat{\mathscr{N}}^p(\check{\mathcal{K}}_h,\mathcal{W}_h)$ from Gal-erkin BEM \eqref{eq:discrete} resp.\ collocation BEM~\eqref{eq:collocation}.
\item[\rm(ii)] Compute indicators $\varrho_h({z})$
for all nodes ${z}\in\NN_h$.
\item[\rm(iii)] Determine a set $\MM_h\subseteq\NN_h$ of minimal cardinality such that
\begin{align}
 \theta\,\varrho_h^2 \le \sum_{{z}\in\MM_h}\varrho_h({z})^2.
\end{align}
\item[\rm(iv)] If both nodes of an element $T\in\mathcal{T}_h$ belong to $\mathcal{M}_h$, $T$  will be marked.
\item[\rm(v)] For all other nodes in $\mathcal{M}_h$, the multiplicity will be increased if it is smaller than $p+1$, otherwise the elements which contain one of these nodes $z\in\mathcal{M}_h$, will be marked.
\item[\rm(vi)] Refine all marked elements $T\in\mathcal{T}_h$ by bisection (insertion of a node with multiplicity one) of the corresponding $\check{T}\in \check{\mathcal{T}}_h$.
Use further bisections to guarantee that the new partition $\mathcal{T}_{H}$ satisfies
\begin{align}\label{eq:kappa small}
\kappa(\check{\mathcal{T}}_{H})\leq 2\kappa(\check{\mathcal{T}}_0).
\end{align}
Update $h\mapsto H$, i.e., replace $\TT_h$ by $\TT_H$.
\end{itemize}
\textbf{Output:} Adaptively generated partition $\TT_h$ with corresponding solution $\phi_h$ and error estimator~$\varrho_h$.\qed
\end{algorithm}

\begin{remark}
{\rm(i)} While $\theta=1$ leads essentially to uniform refinement, $\theta\ll1$ leads to highly adapted partitions. Note that the smaller $\theta$, the more iterations of the adaptive loop are required. In our experiments below, $\theta=0.75$ appeared to be an appropriate compromise which led to optimal convergence behavior.\\
{\rm(ii)} The estimate~\eqref{eq:kappa small} in step~{\rm(iv)} of the adaptive algorithm can be achieved by some extended 1D bisection algorithm from~\cite{feischl}. The latter guarantees that the overall number of elements is bounded by the sum of elements in the initial partition plus the number of marked elements.\qed
\end{remark}

%% file: 03_numerics.tex
\noindent
In this section, we empirically investigate the performance of Algorithm \ref{the algorithm} for Galerkin as well as collocation IGABEM in three typical situations: In  Section~\ref{subsec:singular sol}, the boundary $\Gamma=\partial\Omega$ is closed and the solution exhibits a generic (i.e., geometry induced) singularity.
In Section \ref{subsec:jump sol}, the solution is smooth on $\Gamma=\partial\Omega$, but has certain jumps which require discontinuous ansatz functions.
In Section \ref{subsec:slit problem}, we consider a slit problem.
In all examples, the exact solution is known.
This allows to analyze  the reliability and efficiency of the proposed estimators.

The boundary part $\Gamma$  is  parametrised by a NURBS curve $\gamma$, i.e., the parametrisation has the special form
\begin{equation}\label{eq:NURBS curve}
\gamma(t)=\sum_{i\in \Z} C_i R_{i,p}^{\check{\mathcal{K}}_\gamma,\mathcal{W}_\gamma}(t)
\end{equation}
for all $t\in [a,b]$.
Here,  $p\in \N$ is the polynomial degree, $\check{\mathcal{K}}_\gamma$ and $\mathcal{W}_\gamma$ are knots and weights as in Section \ref{subsec:algorithm} and $(C_i)_{i\in \Z}$ are \textit{control points} in $\R^2$ which are periodic for closed $\Gamma=\partial\Omega$.

We choose the same polynomial degree $p$ for our ansatz spaces $\XX_h= \widehat{\mathscr{N}}^p(\check{\KK}_h,\mathcal{W}_h)$.
For the initial knots and weights, we choose $\check{\KK}_h=\check{\KK}_{\gamma}$ and $\mathcal{W}_h=\mathcal{W}_{\gamma}$.
As the ansatz spaces are nested, it always holds
\begin{align}
\gamma_1,\gamma_2 \in \mathscr{N}^p {(\check{\KK}_h,\mathcal{W}_h)}|_{[a,b]},
\end{align}
where $\gamma_1,\gamma_2$ denote the first resp. second component of $\gamma$.
Therefore, this approach reflects the main idea of isogeometric analysis, i.e., the same space is used for the geometry and for the approximation.
For adaptive Galerkin IGABEM as well as adaptive collocation IGABEM, we compare uniform refinement, where $\mathcal{M}_{h}=\mathcal{N}_{h}$ and hence all elements are refined, and adaptive refinement with $\theta=0.75$.
In addition, we also consider  discontinuous piecewise polynomials.
Note that this is formally only a special case if $w_j=1$ for all weights $w_j$ of  $\mathcal{W}_h$ and $\#z_j=p+1$ for all nodes $z_j\in \mathcal{N}_h$.

As basis for the considered  ansatz spaces, we use \eqref{eq:closed basis} resp. \eqref{eq:open basis}.
To calculate the Galerkin matrix, the collocation matrix, the Faermann error estimator, and the weighted-residual error  estimator, we transform the weakly-singular integrands into a sum of a smooth part and a logarithmically singular part.
Then, we use adapted Gauss quadrature to compute the resulting integrals with appropriate accuracy;
 see \cite[Section 5]{diplarbeit} for details.
For the weighted-residual error estimator \eqref{eq:residual}, we replace $|\omega_h(z)|$ by the length $|\gamma^{-1}(\omega(z))|$, since this eases the calculation.
Note that $|\omega_h(z)|$ $\simeq|\gamma^{-1}(\omega_h(z))|$, where the hidden constants  depend only on the parametrization $\gamma$.

To calculate, the exact error, we proceed as follows:
Let $\phi_{h}^{\rm gal}\in \XX_h$ be the Galerkin approximation with $\boldsymbol{c_h^{\rm gal}}$ the corresponding coefficient vector.
Let  $\phi_{h}^{\rm col}\in \XX_h$ be  the collocation approximation  with $\boldsymbol{c_h^{\rm col}}$ the corresponding coefficient vector.
Let $\boldsymbol{V_h^{\rm gal}}$ be the Galerkin matrix of the $h$-th step.
With the Galerkin orthogonality and the energy norm $\enorm{\phi}^2=\dual{V\phi}{\phi}$, obtained by Aitken's $\Delta^2$-extrapolation, we can compute the energy error as
\begin{align}\label{eq:error calc gal}
\begin{split}
\enorm{\phi-\phi_h^{\rm gal}}^2&=\enorm{\phi}^2-\enorm{\phi_h^{\rm gal}}^2\\
&=\enorm{\phi}^2-\dual{\boldsymbol{V_h^{\rm gal}}\boldsymbol{c_h^{\rm gal}}}{ \boldsymbol{c_h^{\rm gal}}},
\end{split}
\end{align}
resp.
\begin{align}\label{eq:error calc col}
&\enorm{\phi-\phi_h^{\rm col}}^2=\enorm{\phi-\phi_h^{\rm gal}}^2-\enorm{\phi_h^{\rm gal}-\phi_h^{\rm col}}^2\\
&\quad=\enorm{\phi-\phi_h^{\rm gal}}^2-\dual{\boldsymbol{V_h^{\rm gal}}(\boldsymbol{c_h^{\rm gal}}-\boldsymbol{c_h^{\rm col}})}{(\boldsymbol{c_h^{\rm gal}}-\boldsymbol{c_h^{\rm col}})}.\notag
\end{align}

\subsection{Laplace-Dirichlet problem}

In the first two examples, we consider the Laplace-Dirichlet problem
\begin{align}\label{eq:Laplace}
\begin{split}
-\Delta u=0\text{ in }{\Omega}\quad \text{ and }\quad u=g\text{ on } \Gamma_{}
\end{split}
\end{align}
for given Dirichlet data $g\in {H}^{1/2}(\Gamma_{})$ and closed boundary $\Gamma=\partial\Omega$.
The problem is equivalent to the integral equation \eqref{eq:strong} with $f=(K+\sigma)g$, i.e.
\begin{equation}\label{eq:Symmy}
V\phi =(K+\sigma) g\quad\text{on }\Gamma,
\end{equation}
where 
\begin{align}\label{eq:double-layer}
\begin{split}
Kg(x):= -\frac{1}{2\pi}\int_{\Gamma_{}} g(y)\partial_{\nu(y)}\log (|x-y|)  \,dy 
\end{split}
\end{align}
denotes the \textit{double-layer integral operator} and $\sigma(x)=1/2$ for all $x\in \Gamma$ except of the corners, where $\sigma(x)=\alpha/(2\pi)$ with the corresponding interior angle $\alpha$.
The unique solution of \eqref{eq:strong} is the normal derivative $\phi= \partial u/\partial \nu$ of the  solution $u\in H^1(\Omega)$ of \eqref{eq:Laplace}.
For more details, see e.g. \cite[Section 6.3 and 6.6]{steinbach}.

\subsection{Problem with generic singularity}
\label{subsec:singular sol}

As first example, we consider the Laplace-Dirichlet problem \eqref{eq:Laplace} on the pacman geometry
\begin{equation*}
\Omega:=\left\{r(\cos(\beta),\sin(\beta)):0\le r<\frac{1}{10}, \beta \in \left(-\frac{\pi}{2\tau},\frac{\pi}{2\tau}\right)\right\},
\end{equation*}
with $\tau=4/7$; see Figure \ref{fig:pacman}.
The geometry is parametrised on $[0,1]$ by a NURBS curve of degree $p=2$.
We prescribe the exact solution of \eqref{eq:Laplace} as 
\begin{equation*}
u(x,y)=r^{\tau}\cos\left(\tau\beta\right)
\end{equation*}
in polar coordinates $(x,y)=r(\cos \beta,\sin \beta)$.
We consider the corresponding integral equation \eqref{eq:Symmy}.
The normal derivative $\phi=\partial u/\partial \nu$ of $u$ reads 
\begin{equation*}
\phi(x,y)=\begin{pmatrix} \cos(\beta)\cos\left(\tau\beta\right)+\sin(\beta)\sin\left(\tau\beta\right)\\ \sin(\beta)\cos\left(\tau\beta\right)-\cos(\beta)\sin\left(\tau\beta\right)\end{pmatrix}\cdot \nu(x,y) \cdot \tau \cdot r^{\tau-1}
\end{equation*}
and has a generic singularity at the origin.

\begin{figure}[h!] 
\psfrag{pacman (Section 5.3)}[c][c]{}
\psfrag{45}[r][r]{\tiny $45^\circ$}
\begin{center}
\includegraphics[width=0.28\textwidth]{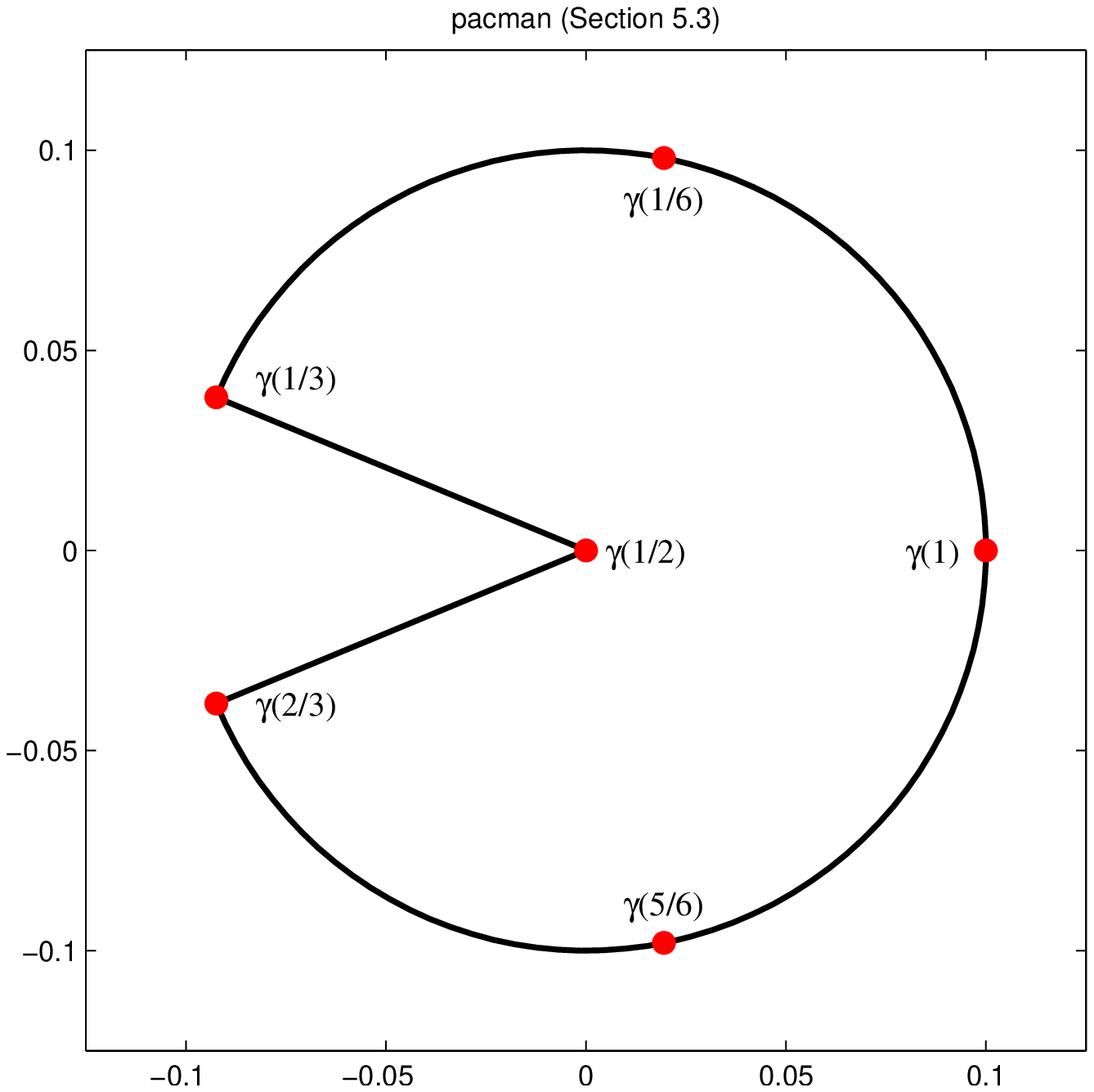}
\end{center}
\caption{Geometry and initial nodes for the experiment from Section~\ref{subsec:singular sol}.} 
\label{fig:pacman}
\end{figure}

\begin{figure}[h!] 
\psfrag{parameter domain}[c][c]{\tiny parameter domain}
\psfrag{solution}[c][c]{\tiny solution} 
\begin{center}
\includegraphics[width=0.28\textwidth]{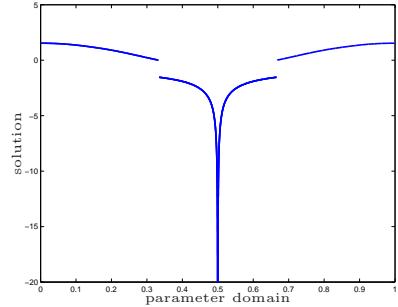}\end{center}
\caption{Experiment with singular solution on pacman geometry from Section \ref{subsec:singular sol}. The singular solution $\phi\circ{\gamma}$ is plotted on the parameter interval, where $0.5$ corresponds to the origin, where $\phi$ is singular.}
\label{fig:pacman solution}
\end{figure}

\begin{figure}[h!]
\psfrag{error and estimator}[c][c]{\tiny error and estimator}
\psfrag{number of knots N}[c][c]{\tiny number of knots $N$}
\psfrag{O(72)}[r][r]{\tiny$\mathcal{O}(N^{-7/2})$}
\psfrag{O(47)}[l][l]{\tiny$\mathcal{O}(N^{-4/7})$}
\psfrag{eta, unif.}[l][l]{\tiny $\eta_h$, unif.}
\psfrag{mu, unif.}[l][l]{\tiny $\mu_h$, unif.}
\psfrag{error, unif.}[l][l]{\tiny  error, unif.}
\psfrag{eta, eta-ad.}[l][l]{\tiny $\eta_h$, $\eta_h$-ad.}
\psfrag{error, eta-ad.}[l][l]{\tiny  error, $\eta_h$-ad.}
\psfrag{mu, mu-ad}[l][l]{\tiny $\mu_h$, $\mu_h$-ad.}
\psfrag{error, mu-ad.}[l][l]{\tiny  error, $\mu_h$-ad.}
\psfrag{Gal1}[c][c]{\tiny Galerkin}
\psfrag{Gal0}[c][c]{\tiny collocation}
\begin{center}
\includegraphics[width=0.35\textwidth]{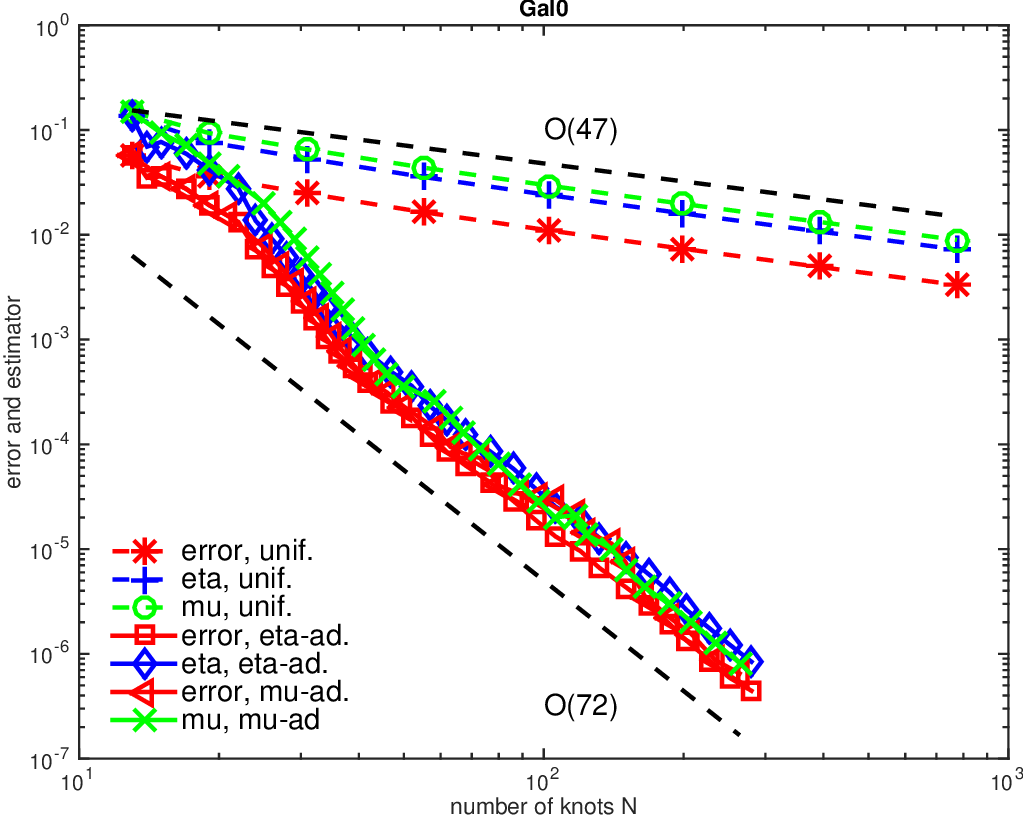}\vspace{2.5mm}
\includegraphics[width=0.35\textwidth]{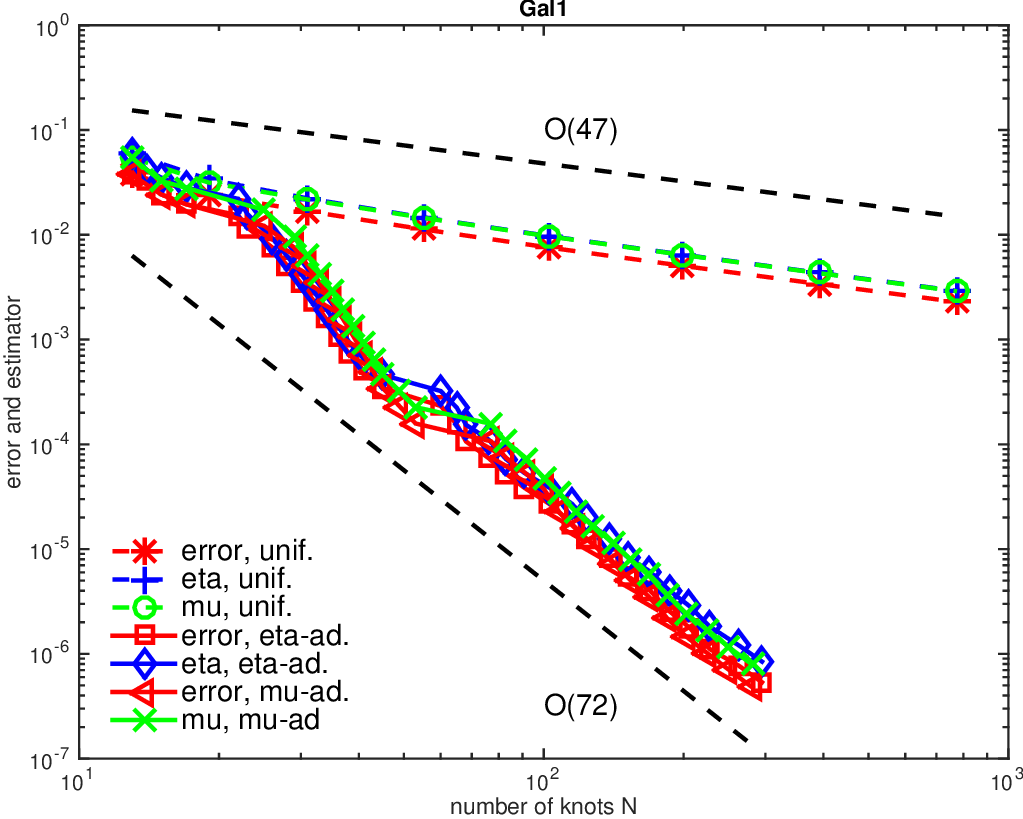}
\end{center}\caption{Experiment with singular solution on pacman geometry from Section \ref{subsec:singular sol}. Error and estimator are plotted versus the number of knots~$N$.
Uniform, $\eta_h$-driven and $\mu_h$-driven refinement is considered.} 
\label{fig:pacman convergence}
\end{figure}

\begin{figure}[h!]
\psfrag{ratio}[c][c]{\tiny efficiency index}
\psfrag{number of knots N}[c][c]{\tiny number of knots $N$}
\psfrag{rho=eta, unif.}[l][l]{\tiny $\rho_h=\eta_h$, unif.}
\psfrag{rho=mu, unif.}[l][l]{\tiny $\rho_h=\mu_h$, unif.}
\psfrag{rho=eta, ad.}[l][l]{\tiny $\rho_h=\eta_h$, ad.}
\psfrag{rho=mu, ad.}[l][l]{\tiny $\rho_h=\mu_h$, ad.}
\psfrag{Gal1}[c][c]{\tiny Galerkin}
\psfrag{Gal0}[c][c]{\tiny collocation}
\begin{center}
\includegraphics[width=0.35\textwidth]{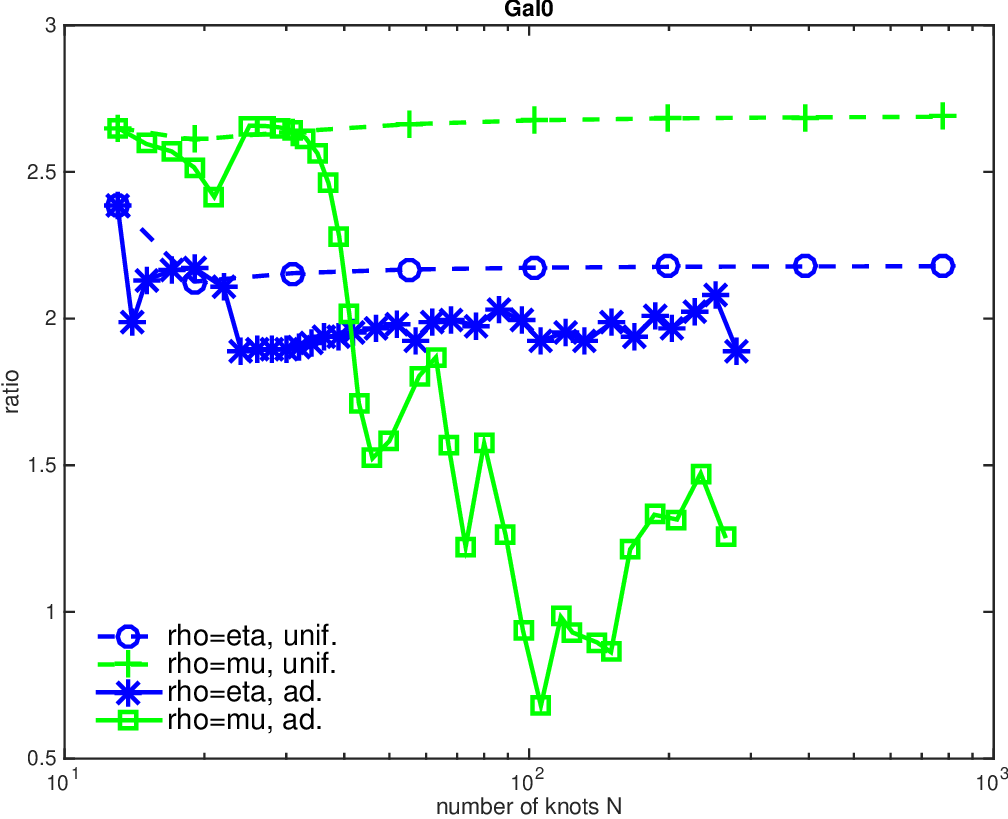}\vspace{2.5mm}
\includegraphics[width=0.35\textwidth]{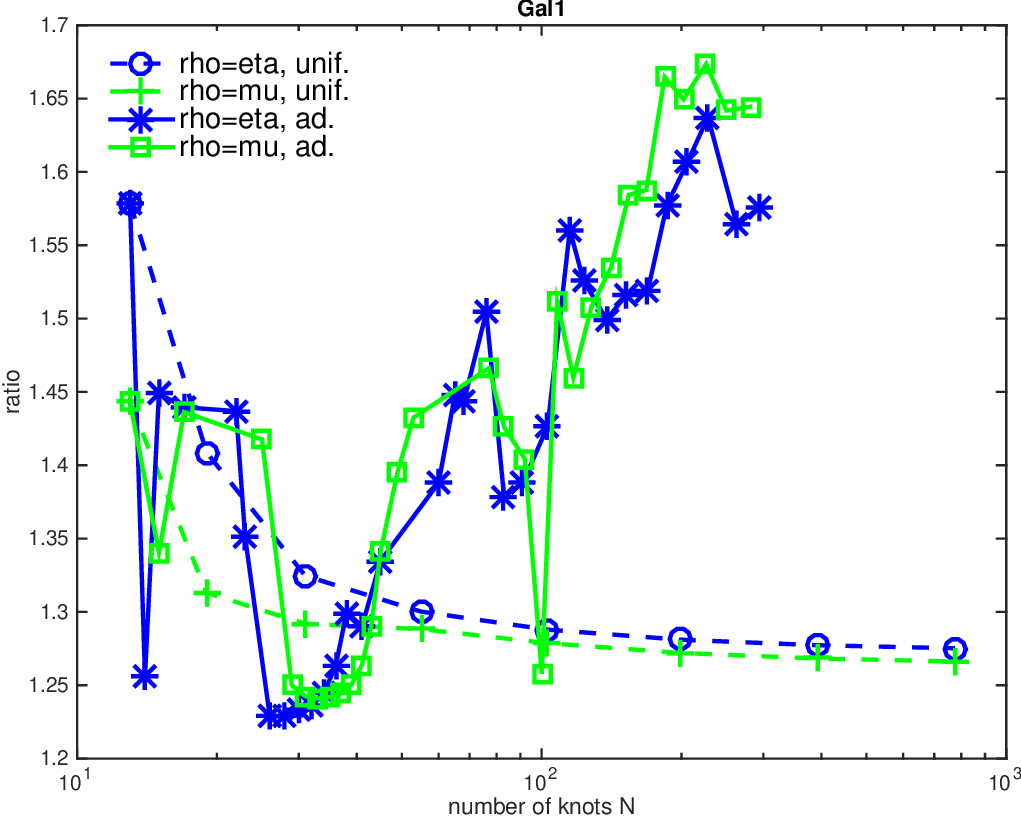}
\end{center}
\caption{Experiment with singular solution on pacman geometry from Section \ref{subsec:singular sol}. 
The plot shows the efficiency indices  $\frac{\rho_h}{\enorm{\phi-\phi_h}}$ for the estimators $\rho_h\in \{\eta_h,\mu_h\}$, where adaptivity is driven by $\rho_h$.} 
\label{fig:pacman ratio}
\end{figure} 


\begin{figure}[h!]
\psfrag{Gal., rho=eta}{\tiny Gal., $\rho_h=\eta_h$}
\psfrag{Gal., rho=mu}{\tiny Gal., $\rho_h=\mu_h$}
\psfrag{coll., rho=eta}{\tiny coll., $\rho_h=\eta_h$}
\psfrag{coll., rho=mu}{\tiny coll., $\rho_h=\mu_h$}
\psfrag{error}[c][c]{\tiny error}
\psfrag{number of knots N}[c][c]{\tiny number of knots $N$}
\psfrag{O(72)}[c][c]{\tiny$\mathcal{O}(N^{-7/2})$}
\psfrag{O(47)}[l][l]{\tiny$\mathcal{O}(N^{-4/7})$}
\begin{center}
\includegraphics[width=0.35\textwidth]{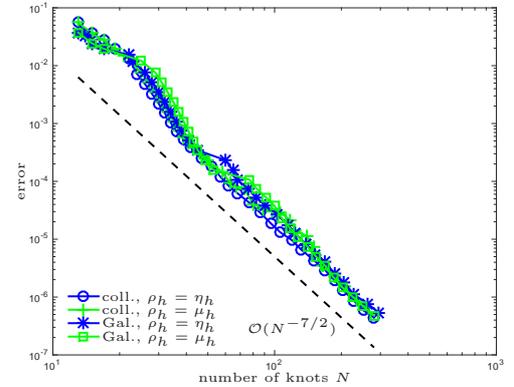}
\end{center}
\caption{Experiment with singular solution on pacman geometry from Section \ref{subsec:singular sol}. The errors from all presented adaptive IGABEM strategies are plotted versus the number of knots~$N$.} 
\label{fig:pacman errors}
\end{figure}

\begin{figure}[h!]
\psfrag{parameter domain}[c][c]{\tiny parameter domain}
\psfrag{Cars}[l][l]{\tiny $\eta$}
\psfrag{Faer}[l][l]{\tiny $\mu$}
\psfrag{error}[l][l]{\tiny  error}
\psfrag{Gal1est2}[c][c]{\tiny Galerkin with $\rho=\mu$}
\psfrag{Gal0est2}[c][c]{}
\psfrag{Gal1est1}[c][c]{\tiny Galerkin with $\rho=\eta$}
\psfrag{Gal0est1}[c][c]{\tiny collocation with $\rho=\eta$}
\begin{center}
\includegraphics[width=0.3\textwidth]{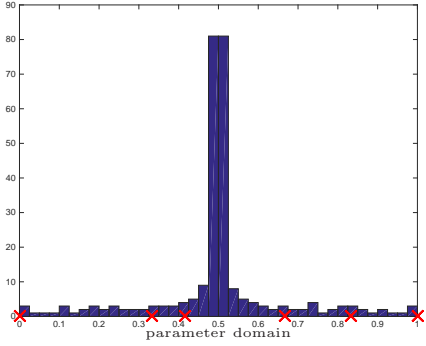}
\end{center}
\caption{Experiment with singular solution on pacman geometry from Section \ref{subsec:singular sol}.
Histogram of number of knots over the parameter domain. Knots with maximal multiplicity $p+1=3$ are marked.}
\label{fig:pacman knots}
\end{figure}

\begin{figure}[h!]
\psfrag{Gal., pp, unif.}{\tiny pp, unif. Gal.}
\psfrag{Gal., pp, ad.}{\tiny pp, ad. Gal.}
\psfrag{Gal., IGA, ad.}{\tiny IGA, ad. Gal.}
\psfrag{Gal., IGA, unif.}{\tiny IGA, unif. Gal}
\psfrag{coll., IGA, ad.}{\tiny IGA, ad. coll.}
\psfrag{coll., IGA, unif.}{\tiny IGA, unif. coll.}
\psfrag{error}[c][c]{\tiny error}
\psfrag{number of knots N}[c][c]{\tiny number of knots $N$}
\psfrag{O(72)}[c][c]{\tiny$\mathcal{O}(N^{-7/2})$}
\psfrag{O(47)}[l][l]{\tiny$\mathcal{O}(N^{-4/7})$}
\begin{center}
\includegraphics[width=0.35\textwidth]{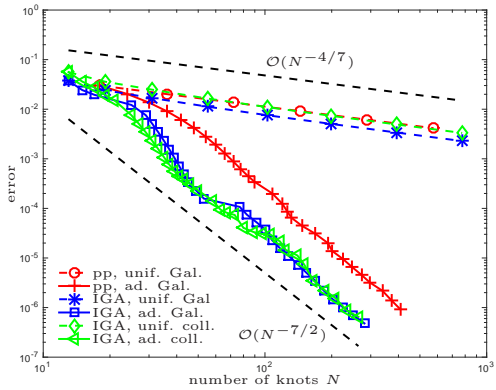}
\end{center}
\caption{Experiment with singular solution on pacman geometry from Section \ref{subsec:singular sol}. The errors from uniform/adaptive BEM with discontinuous piecewise polynomials and uniform/adaptive IGABEM are plotted versus the number of knots~$N$.} 
\label{fig:pacman compare}
\end{figure}

In Figure \ref{fig:pacman solution}, the solution $\phi$ is plotted over the parameter domain.
The singularity is located at $t=1/2$ and two jumps are located at $t=1/3$ rep. $t=2/3$.

  In Figure \ref{fig:pacman convergence}, 
  error and error estimators are plotted.
All values are plotted in a double logarithmic scale such that the experimental convergence rates are visible as the slope of the corresponding curves.
Since the solution lacks regularity, uniform refinement leads to the suboptimal rate $\mathcal{O}(N^{-4/7})$ for the energy error, whereas adaptive refinement leads to the optimal rate $\mathcal{O}(N^{-7/2})$.
In each case, the curves for the two different estimators $\eta_h$ and $\mu_h$ and the error are parallel.
In Figure \ref{fig:pacman ratio}, we plot the ratios $\eta_h/\enorm{\phi-\phi_h}$ resp. $\mu_h/\enorm{\phi-\phi_h}$.
Throughout, these ratios stay between $0.5$ and $2.7$ which underlines an accurate error estimation for both error estimators.
Figure \ref{fig:pacman errors} shows the errors of all considered adaptive IGABEM strategies.
We observe a very similar behaviour.

For adaptive refinement, Figure \ref{fig:pacman knots} provides  a histogram of the knots in $[a,b]$ of the last refinement step for collocation IGABEM with $\rho_h=\mu_h$, for the other adaptive strategies, the output looks similar (not displayed).
We see that the algorithm mainly refines the mesh around the singularity at $t=1/2$.
Additionally, the multiplicity at the jump points $t=1/3$ and $t=2/3$ appears to be maximal so that the discrete solution $\phi_h$  also mimics the discontinuities of the exact solution $\phi$.

In Figure \ref{fig:pacman compare}, we finally compare standard BEM with discontinuous piecewise polynomials against IGABEM.
For the error estimation we use the weighted-residual estimator $\mu_h$.
The output looks similar if $\eta_h$ is used instead (not displayed).
All approaches show similar convergence rates, however we clearly observe better multiplicative constants for Galerkin IGABEM and collocation IGABEM than for standard BEM.
\subsection{Adaptive IGABEM for problem with jump solution}
\label{subsec:jump sol}
As second example, we consider the Laplace-Dirichlet problem \eqref{eq:Laplace} on the square $\Omega=[0,1/2]^2$; see Figure \ref{fig:square}.
The geometry is parametrised on $[0,1]$ by a NURBS curve of degree $p=1$.

\begin{figure}[h!] 
\psfrag{square (Section 5.2)}[c][c]{}
\begin{center}
\includegraphics[width=0.25 \textwidth]{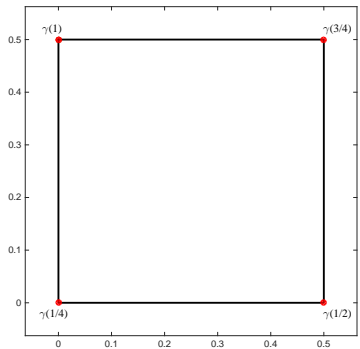}
\end{center}
\caption{Geometry and initial nodes for the experiments from Section~\ref{subsec:jump sol}.} 
\label{fig:square}
\end{figure}
We prescribe the exact solution of \eqref{eq:Laplace} as 
\begin{equation*}
u(x,y)=\sinh(2\pi x)\cos(2\pi y).
\end{equation*}
We consider the corresponding integral equation \eqref{eq:Symmy}.
The normal derivative $\phi=\partial u / \partial \nu$ of $u$ reads 
\begin{equation*}
\phi(x,y)=2\pi\begin{pmatrix} \cosh(2\pi x)\cos(2\pi y) \\ \sinh(2\pi x)\cos(2\pi y)\end{pmatrix}\cdot \nu(x,y).
\end{equation*}
It is smooth up to four jumps as can be seen in Figure~\ref{fig:square solution}.


  In Figure \ref{fig:square convergence} 
  we plot error and error estimators.
 The solution $\phi\circ\gamma$ has jumps at the points $t=1/4$, $t=1/2$, $t=3/4$ and $t=1$ resp. $t=0$.
As the knots $\check{\KK}_\gamma$ used for the parametrisation of $\Gamma$ all have multiplicity one, the functions of the isogeometric start approximation space are continuous at the points $t=1/4$, $t=1/2$ and $t=3/4$.
 Uniform refinement, where only $h$-refinement takes place, leads to the suboptimal rate $\mathcal{O}(N^{-1})$ for the energy error, whereas adaptive refinement increases the knot multiplicity at these problematic points and  leads again to the optimal rate $\mathcal{O}(N^{-5/2})$.
In Figure \ref{fig:square ratio}, we plot the efficiency indices $\eta_h/\enorm{\phi-\phi_h}$ resp. $\mu_h/\enorm{\phi-\phi_h}$. 
Throughout, these ratios stay between $0.1$ and $2.2$.
Figure~\ref{fig:square errors} shows the errors of all considered adaptive IGABEM strategies.
We observe that $\eta_h$ leads to slightly better results than $\mu_h$, while there appears to be almost no difference between Galerkin IGABEM and collocation IGABEM.


In Figure \ref{fig:square compare},  standard BEM with discontinuous piecewise polynomials is compared against IGABEM.
For adaptivity, we use the weighted-residual estimator $\mu_h$.
The output looks similar if the estimator $\eta_h$ is used  (not displayed).
We observe that in this example uniform standard BEM is superior to uniform IGABEM.
This is of course due to the fact that standard BEM uses ansatz spaces which are discontinuous at the jumps of $\phi$.
However, with the use of adaptive multiplicity increase this is fixed as can be seen in the convergence plot, where we again see that adaptive IGABEM leads to better results than adaptive standard BEM.
It is also interesting that adaptive standard BEM converges with a better multiplicative constant than uniform standard BEM. 
This is due to the fact that the solution is zero on $[1/4,1/2]$ and $[3/4,1]$, wherefore the adaptive algorithm uses only few elements in this area.

\begin{figure}[h] 
\psfrag{parameter domain}[c][c]{\tiny parameter domain}
\psfrag{solution}[c][c]{\tiny solution} 
\begin{center}\includegraphics[width=0.32\textwidth]{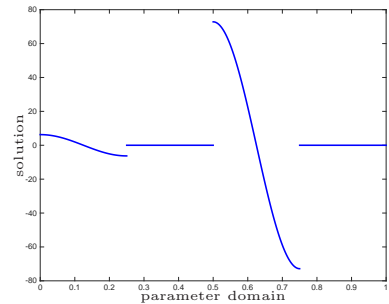}\end{center}
\caption{Experiment with jump solution on square from Section~\ref{subsec:jump sol}. The solution $\phi\circ{\gamma}$ is plotted on the parameter interval.}
\label{fig:square solution}
\end{figure}

\begin{figure}[h!]
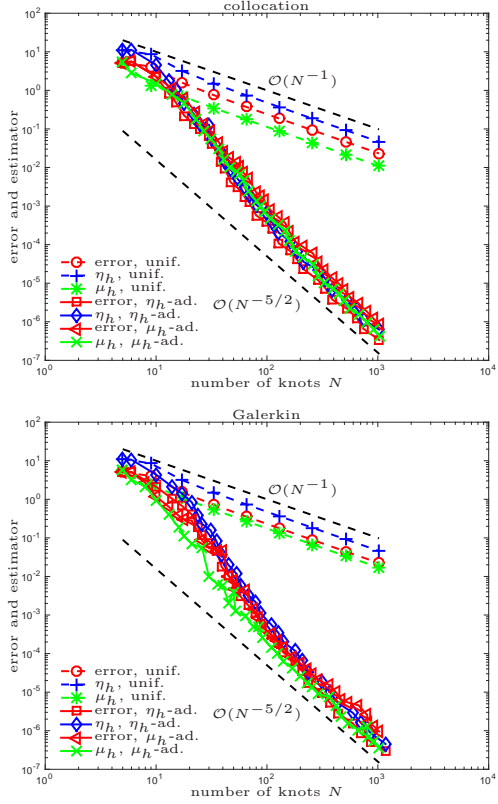

\psfrag{eta, unif.}[l][l]{\tiny $\eta_h$, unif.}
\psfrag{mu, unif.}[l][l]{\tiny $\mu_h$, unif.}
\psfrag{error, unif.}[l][l]{\tiny  error, unif.}
\psfrag{eta, eta-ad.}[l][l]{\tiny $\eta_h$, $\eta_h$-ad.}
\psfrag{error, eta-ad.}[l][l]{\tiny  error, $\eta_h$-ad.}
\psfrag{mu, mu-ad}[l][l]{\tiny $\mu_h$, $\mu_h$-ad.}
\psfrag{error, mu-ad.}[l][l]{\tiny  error, $\mu_h$-ad.}
\psfrag{Gal1}[c][c]{\tiny Galerkin}
\psfrag{Gal0}[c][c]{\tiny collocation}
\psfrag{error and estimator}[c][c]{\tiny error and estimator}
\psfrag{number of knots N}[c][c]{\tiny number of knots $N$}
\psfrag{O(52)}[r][r]{\tiny$\mathcal{O}(N^{-5/2})$}
\psfrag{O(1)}[l][l]{\tiny$\mathcal{O}(N^{-1})$}
\begin{center}\includegraphics[width=0.35\textwidth]{geo_5IGA_1Gal_0.eps}\vspace{2.5mm}
\includegraphics[width=0.35\textwidth]{geo_5IGA_1Gal_1.eps}\end{center}
\caption{Experiment with jump solution on square from Section~\ref{subsec:jump sol}. Error and estimator are plotted versus the number of knots~$N$.
Uniform, $\eta_h$-driven and $\mu_h$-driven refinement is considered.} 
\label{fig:square convergence}
\end{figure}

\begin{figure}[h!]
\psfrag{ratio}[c][c]{\tiny ratio}
\psfrag{number of knots N}[c][c]{\tiny number of knots $N$}
\psfrag{rho=eta, unif.}[l][l]{\tiny $\rho_h=\eta_h$, unif.}
\psfrag{rho=mu, unif.}[l][l]{\tiny $\rho_h=\mu_h$, unif.}
\psfrag{rho=eta, ad.}[l][l]{\tiny $\rho_h=\eta_h$, ad.}
\psfrag{rho=mu, ad.}[l][l]{\tiny $\rho_h=\mu_h$, ad.}
\psfrag{Gal1}[c][c]{\tiny Galerkin}
\psfrag{Gal0}[c][c]{\tiny collocation}
\begin{center}
\includegraphics[width=0.35\textwidth]{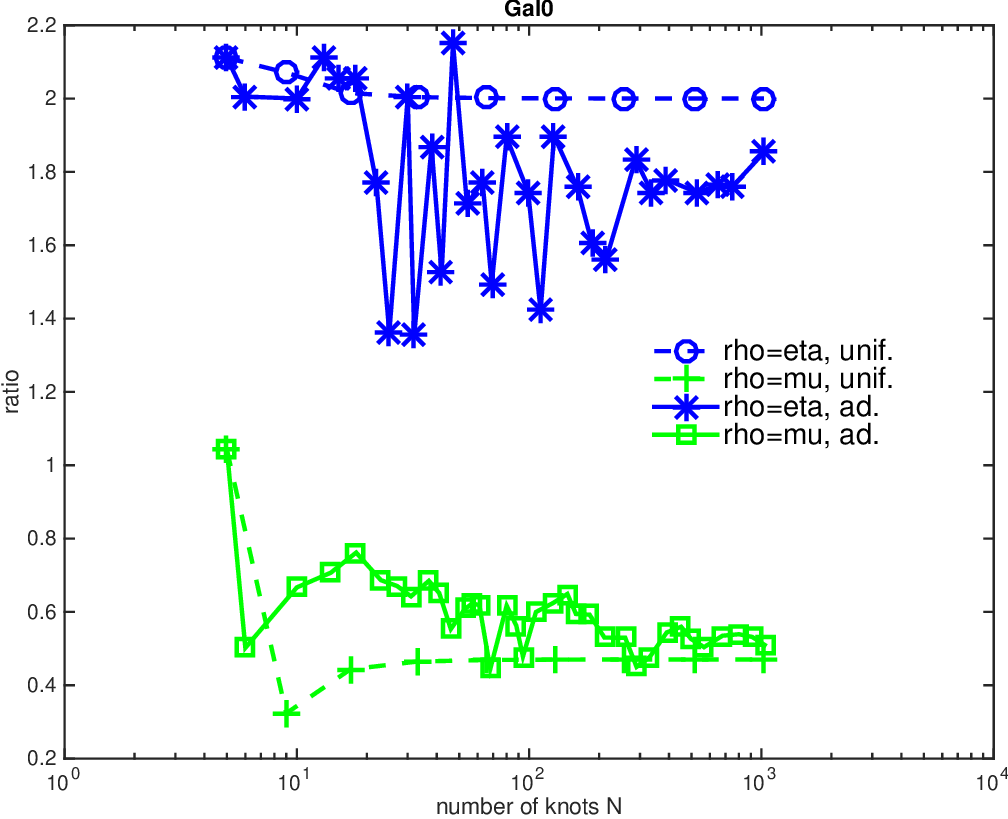}\vspace{2.5mm}
\includegraphics[width=0.35\textwidth]{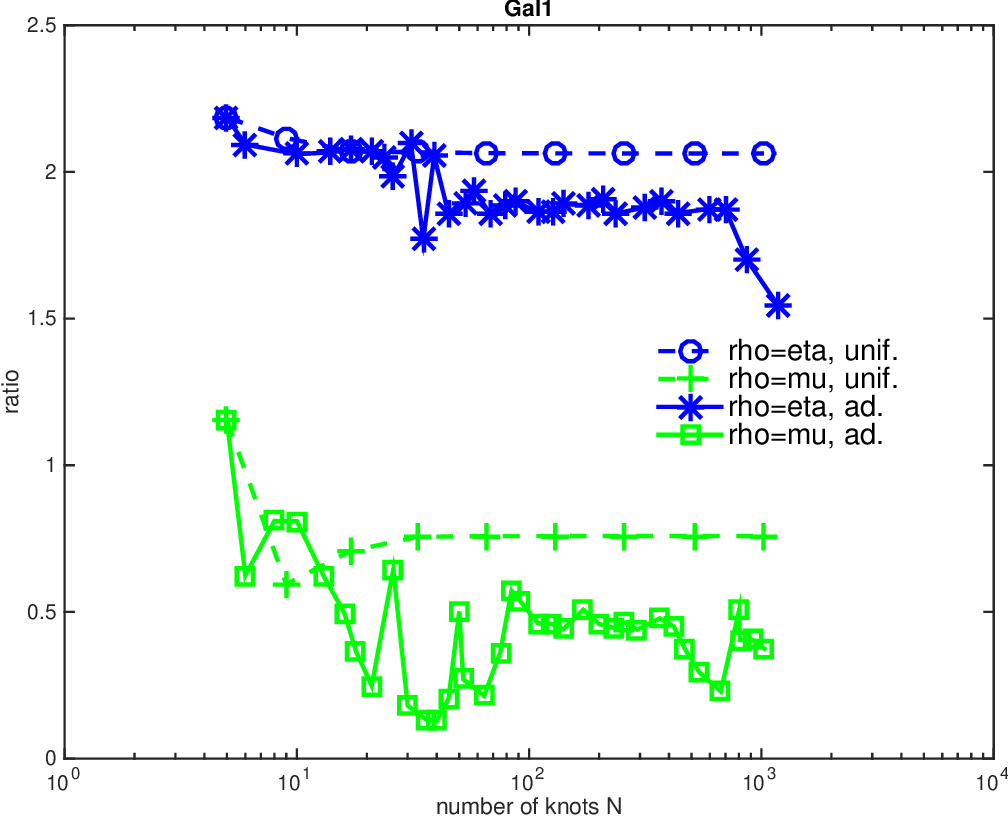}
\end{center}
\caption{Experiment with jump solution on square from Section~\ref{subsec:jump sol}. The plot shows the efficiency indices  $\frac{\rho_h}{\enorm{\phi-\phi_h}}$ for the estimators $\rho_h\in \{\eta_h,\mu_h\}$, where adaptivity is driven by $\rho_h$.} 
\label{fig:square ratio}
\end{figure} 


\begin{figure}[h!]
\psfrag{Gal., rho=eta}{\tiny Gal., $\rho_h=\eta_h$}
\psfrag{Gal., rho=mu}{\tiny Gal., $\rho_h=\mu_h$}
\psfrag{coll., rho=eta}{\tiny coll., $\rho_h=\eta_h$}
\psfrag{coll., rho=mu}{\tiny coll., $\rho_h=\mu_h$}
\psfrag{error}[c][c]{\tiny error}
\psfrag{number of knots N}[c][c]{\tiny number of knots $N$}
\psfrag{O(52)}[r][r]{\tiny$\mathcal{O}(N^{-5/2})$}
\psfrag{O(47)}[l][l]{\tiny$\mathcal{O}(N^{-4/7})$}
\begin{center}
\includegraphics[width=0.35\textwidth]{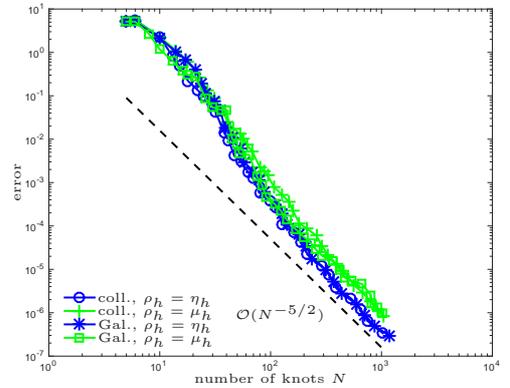}
\end{center}
\caption{Experiment with jump solution on square from Section~\ref{subsec:jump sol}. The errors from all presented adaptive IGABEM strategies are plotted versus the number of knots~$N$.} 
\label{fig:square errors}
\end{figure}

\begin{figure}[h!]
\psfrag{Gal., pp, unif.}{\tiny pp, unif. Gal.}
\psfrag{Gal., pp, ad.}{\tiny pp, ad. Gal.}
\psfrag{Gal., IGA, ad.}{\tiny IGA, ad. Gal.}
\psfrag{Gal., IGA, unif.}{\tiny IGA, unif. Gal}
\psfrag{coll., IGA, ad.}{\tiny IGA, ad. coll.}
\psfrag{coll., IGA, unif.}{\tiny IGA, unif. coll.}
\psfrag{error}[c][c]{\tiny error}
\psfrag{number of knots N}[c][c]{\tiny number of knots $N$}
\psfrag{O(52)}[r][r]{\tiny$\mathcal{O}(N^{-5/2})$}
\psfrag{O(1)}[l][l]{\tiny$\mathcal{O}(N^{-1})$}
\psfrag{Gal1est2}{\tiny Galerkin with $\rho=\mu$}
\psfrag{Gal0est2}{\tiny collocation with $\rho=\mu$}
\psfrag{Gal1est1}{\tiny Galerkin with $\rho=\eta$}
\psfrag{Gal0est1}{\tiny collocation with $\rho=\eta$}
\begin{center}
\includegraphics[width=0.35\textwidth]{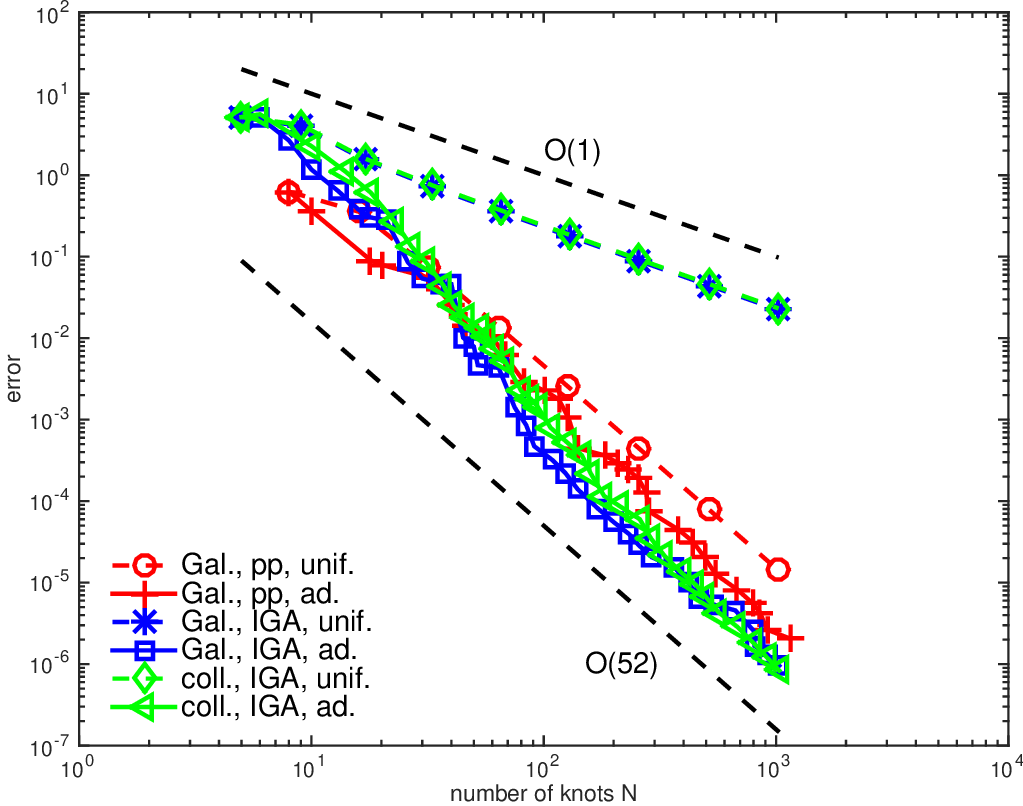}
\end{center}
\caption{Experiment with jump solution on square from Section~\ref{subsec:jump sol}. The errors from uniform BEM with discontinuous piecewise polynomials and uniform/adaptive IGABEM are plotted versus the number of knots~$N$.} 
\label{fig:square compare}
\end{figure} 

\subsection{Adaptive IGABEM for slit problem}
\label{subsec:slit problem}

As last example, we consider a crack problem on the slit $\Gamma= [-1,1]\times\{0\}$.
We parametrize $\Gamma$ by a NURBS curve of degree $p=1$. 
 For $f(x,0):=-x/2$ and the single-layer operator $V$, the exact solution of \eqref{eq:strong} reads
\begin{align*}
\phi(x,0)=\frac{-x}{\sqrt{1-x^2}}.
\end{align*}
Note that $\phi\in \H^{-\varepsilon}(\Gamma)\setminus L^2(\Gamma)$ for all $\varepsilon>0$ and that $\phi$ has singularities at the tips $x=\pm 1$.


\begin{figure}[h!] 
\psfrag{slit (Section 5.4)}[c][c]{}
\begin{center}
\includegraphics[width=0.28 \textwidth]{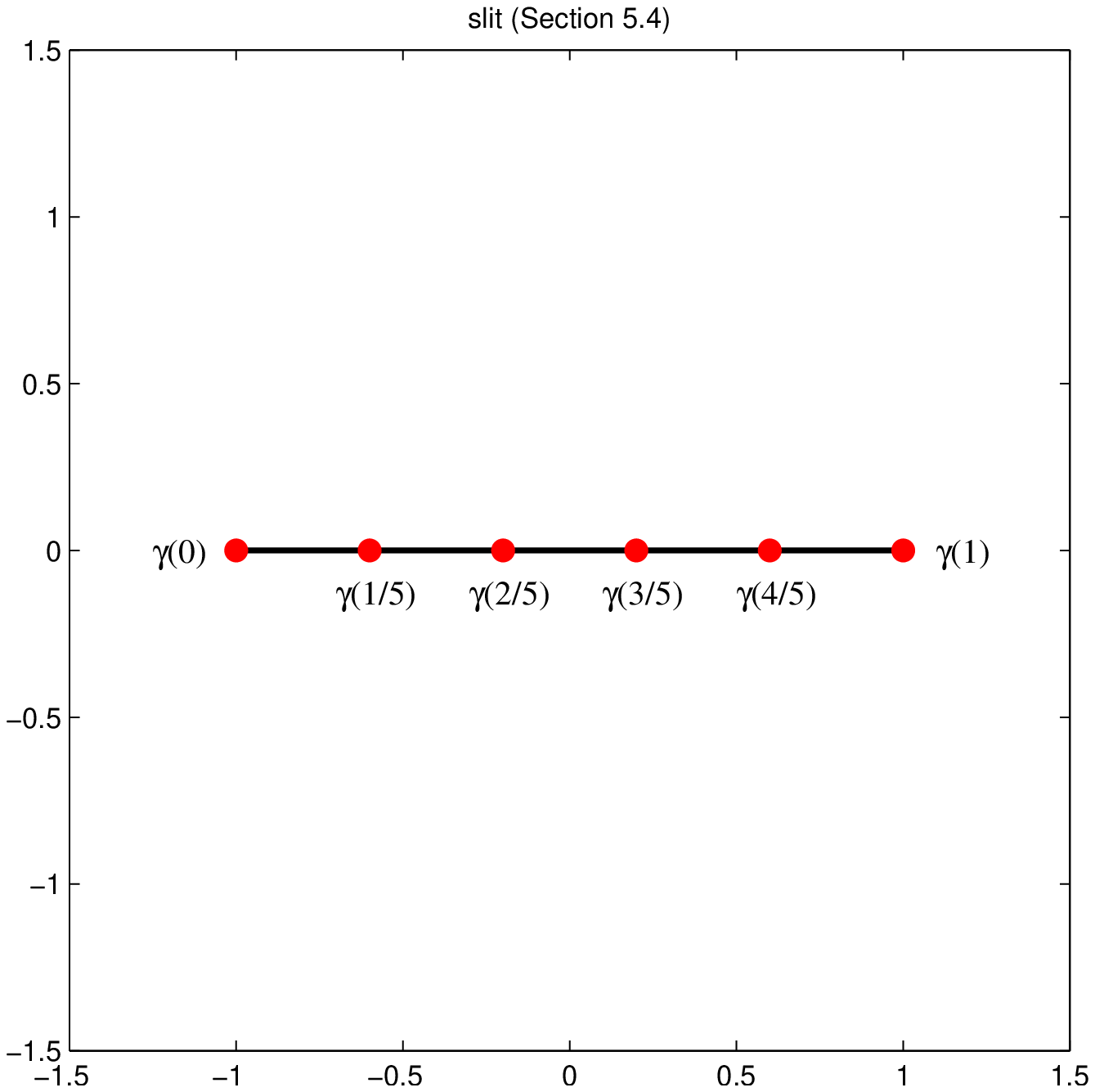}
\caption{Geometry and initial nodes for the experiment from Section~\ref{subsec:slit problem}.} 
\label{fig:geometries}
\end{center}
\end{figure}

\begin{figure}[h!]
\psfrag{error and estimator}[c][c]{\tiny error and estimator}
\psfrag{number of knots N}[c][c]{\tiny number of knots $N$}
\psfrag{O(52)}[r][r]{\tiny$\mathcal{O}(N^{-5/2})$}
\psfrag{O(12)}[l][l]{\tiny$\mathcal{O}(N^{-1/2})$}
\psfrag{estimator and error}{\tiny estimator and error}
\psfrag{number of knots N}{\tiny number of knots $N$}
\psfrag{eta, unif.}[l][l]{\tiny $\eta_h$, unif.}
\psfrag{mu, unif.}[l][l]{\tiny $\mu_h$, unif.}
\psfrag{error, unif.}[l][l]{\tiny  error, unif.}
\psfrag{eta, eta-ad.}[l][l]{\tiny $\eta_h$, $\eta_h$-ad.}
\psfrag{error, eta-ad.}[l][l]{\tiny  error, $\eta_h$-ad.}
\psfrag{mu, mu-ad}[l][l]{\tiny $\mu_h$, $\mu_h$-ad.}
\psfrag{error, mu-ad.}[l][l]{\tiny  error, $\mu_h$-ad.}
\psfrag{Gal1}[c][c]{\tiny Galerkin}
\psfrag{Gal0}[c][c]{\tiny collocation}
\begin{center}\includegraphics[width=0.35\textwidth]{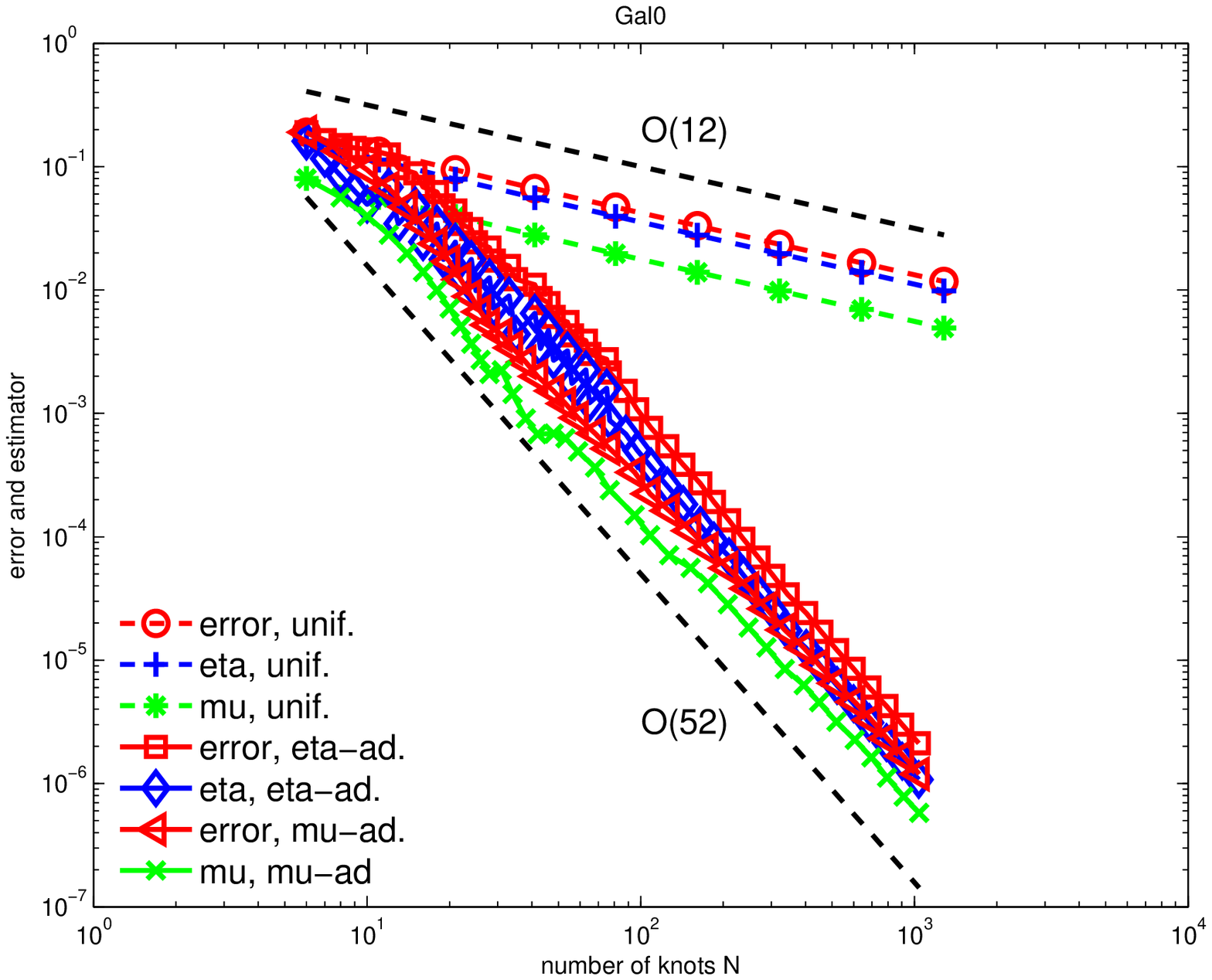}\vspace{2.5mm}
\includegraphics[width=0.35\textwidth]{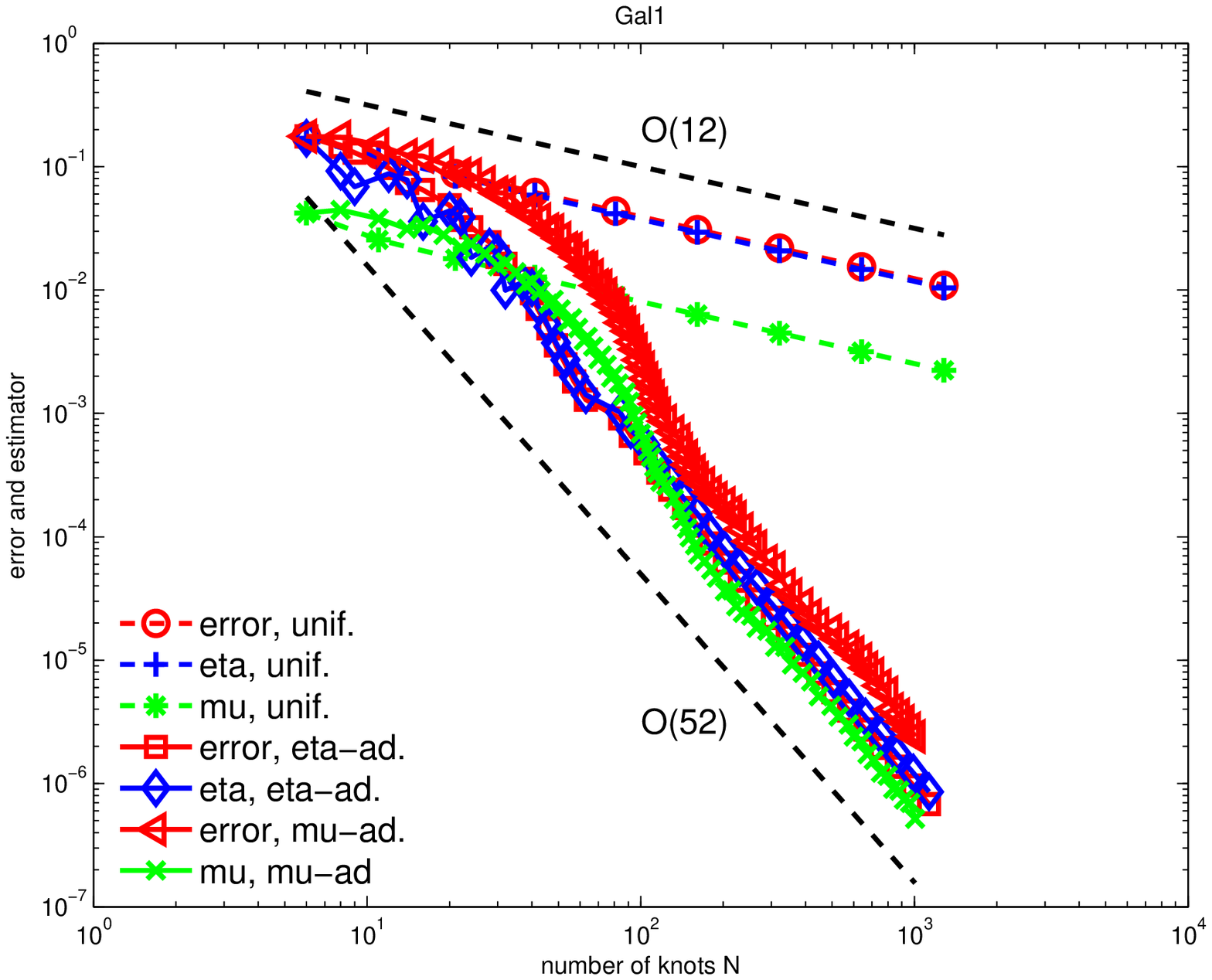}
\end{center}
\caption{Experiment with singular solution on slit from Section~\ref{subsec:slit problem}. Error and estimator are plotted versus the number of knots~$N$. Uniform, $\eta_h$-driven and $\mu_h$-driven refinement is considered.} 
\label{fig:slit convergence}
\end{figure} 

\begin{figure}[h!]
\psfrag{Gal., rho=eta}{\tiny Gal., $\rho_h=\eta_h$}
\psfrag{Gal., rho=mu}{\tiny Gal., $\rho_h=\mu_h$}
\psfrag{coll., rho=eta}{\tiny coll., $\rho_h=\eta_h$}
\psfrag{coll., rho=mu}{\tiny coll., $\rho_h=\mu_h$}
\psfrag{error}[c][c]{\tiny error}
\psfrag{number of knots N}[c][c]{\tiny number of knots $N$}
\psfrag{O(52)}[r][r]{\tiny$\mathcal{O}(N^{-5/2})$}
\psfrag{O(47)}[l][l]{\tiny$\mathcal{O}(N^{-4/7})$}
\begin{center}
\includegraphics[width=0.35\textwidth]{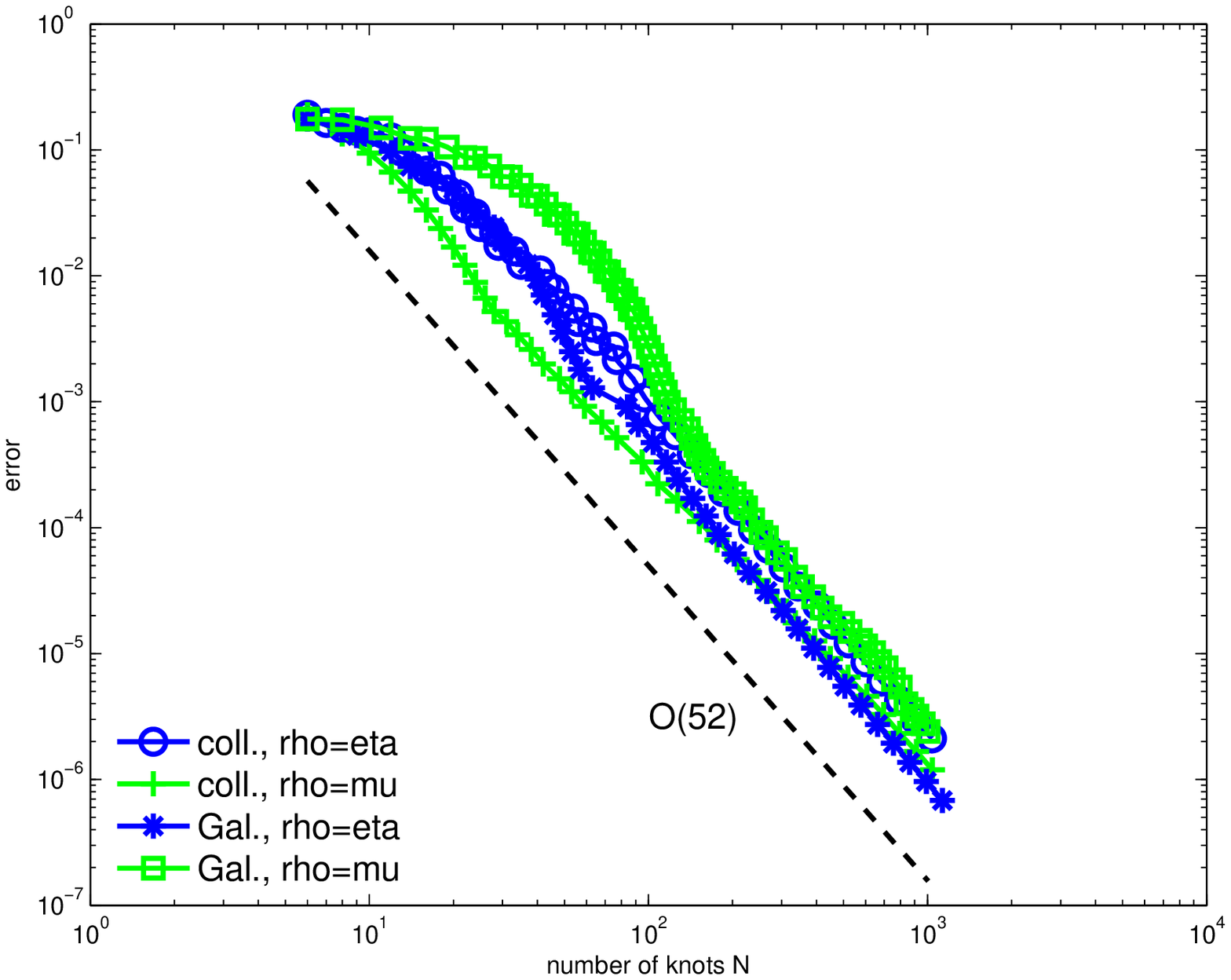}
\end{center}
\caption{Experiment with singular solution on slit from Section~\ref{subsec:slit problem}. The errors from all presented adaptive IGABEM strategies are plotted versus the number of knots~$N$.} 
\label{fig:slit errors}
\end{figure}

\begin{figure}[h!]
\psfrag{ratio}[c][c]{\tiny ratio}
\psfrag{number of knots N}[c][c]{\tiny number of knots $N$}
\psfrag{rho=eta, unif.}[l][l]{\tiny $\rho_h=\eta_h$, unif.}
\psfrag{rho=mu, unif.}[l][l]{\tiny $\rho_h=\mu_h$, unif.}
\psfrag{rho=eta, ad.}[l][l]{\tiny $\rho_h=\eta_h$, ad.}
\psfrag{rho=mu, ad.}[l][l]{\tiny $\rho_h=\mu_h$, ad.}
\psfrag{Gal1}[c][c]{\tiny Galerkin}
\psfrag{Gal0}[c][c]{\tiny collocation}
\begin{center}
\includegraphics[width=0.35\textwidth]{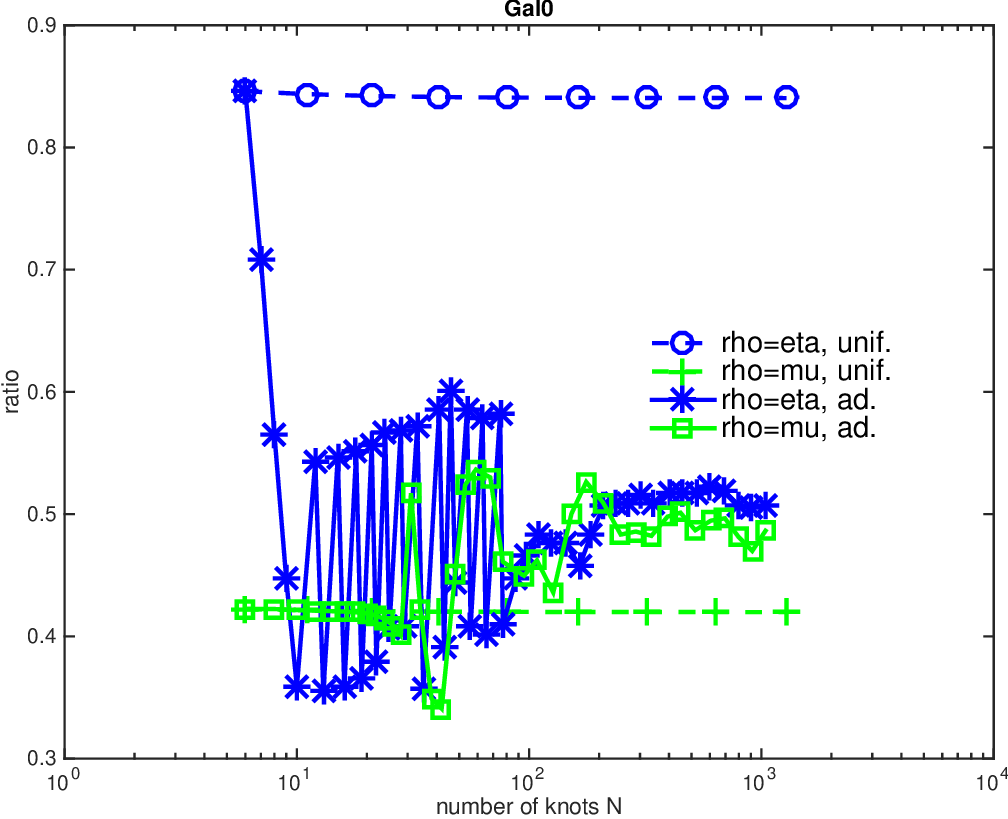}\vspace{2.5mm}
\includegraphics[width=0.35\textwidth]{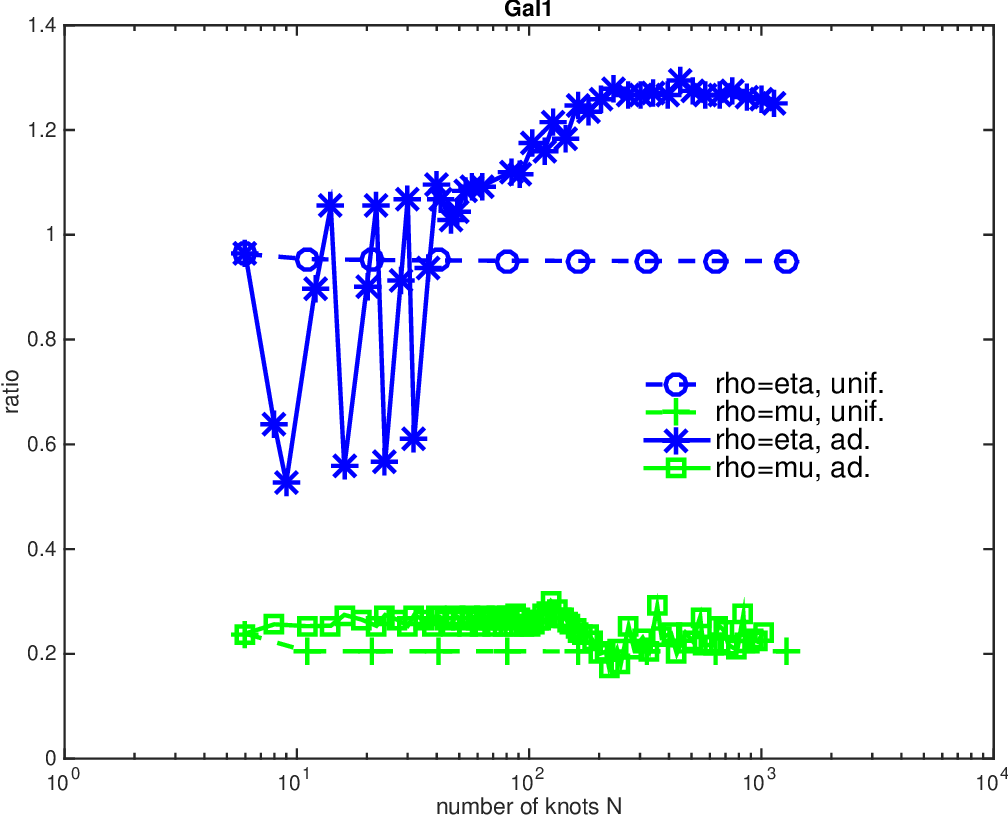}
\end{center}
\caption{Experiment with singular solution on slit from Section~\ref{subsec:slit problem}. The plot shows the efficiency indices  $\frac{\rho_h}{\enorm{\phi-\phi_h}}$ for the estimators $\rho_h\in \{\eta_h,\mu_h\}$, where adaptivity is driven by $\rho_h$.} 
\label{fig:slit ratio}
\end{figure} 

%

\begin{figure}[h!]
\psfrag{Gal., pp, unif.}{\tiny pp, unif. Gal.}
\psfrag{Gal., pp, ad.}{\tiny pp, ad. Gal.}
\psfrag{Gal., IGA, ad.}{\tiny IGA, ad. Gal.}
\psfrag{Gal., IGA, unif.}{\tiny IGA, unif. Gal}
\psfrag{coll., IGA, ad.}{\tiny IGA, ad. coll.}
\psfrag{coll., IGA, unif.}{\tiny IGA, unif. coll.}
\psfrag{error}[c][c]{\tiny error}
\psfrag{number of knots N}[c][c]{\tiny number of knots $N$}
\psfrag{O(52)}[c][l]{\tiny$\mathcal{O}(N^{-5/2})$}
\psfrag{O(12)}[l][l]{\tiny$\mathcal{O}(N^{-1/2})$}
\psfrag{Gal1est2}{\tiny Galerkin with $\rho=\mu$}
\psfrag{Gal0est2}{\tiny collocation with $\rho=\mu$}
\psfrag{Gal1est1}{\tiny Galerkin with $\rho=\eta$}
\psfrag{Gal0est1}{\tiny collocation with $\rho=\eta$}
\begin{center}
\includegraphics[width=0.35\textwidth]{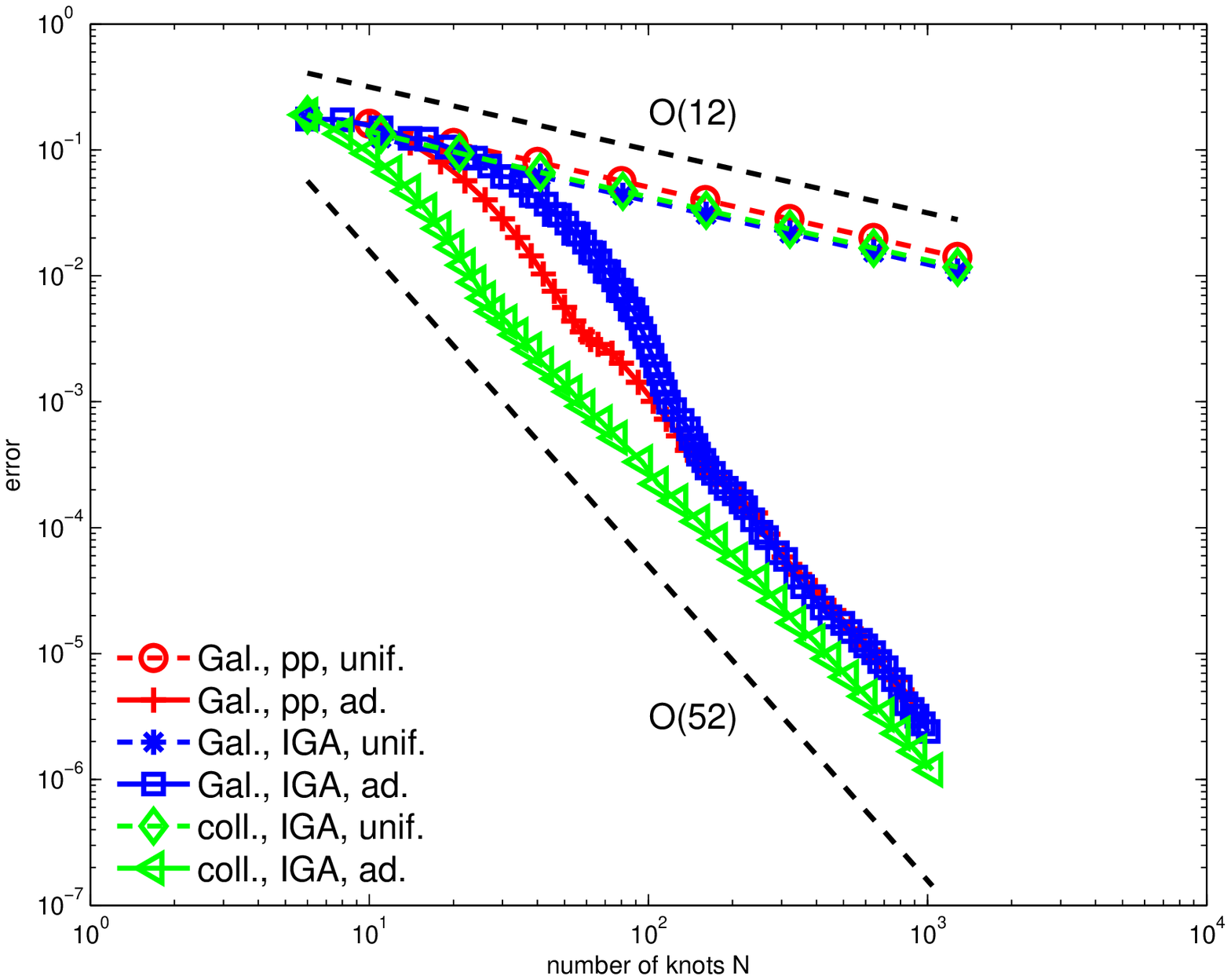}
\end{center}
\caption{Experiment with singular solution on slit from Section~\ref{subsec:slit problem}. The errors from uniform BEM with discontinuous piecewise polynomials and uniform/adaptive IGABEM are plotted versus the number of knots~$N$.} 
\label{fig:slit compare}
\end{figure}

In Figure \ref{fig:slit convergence},   error and error estimators for the uniform and  for the adaptive approach are plotted.
The error is obtained via \eqref{eq:error calc gal} resp. \eqref{eq:error calc col}, where $\enorm{\phi}^2=\pi/4$ is computed analytically.
Since the solution lacks regularity, uniform refinement leads to the suboptimal rate $\mathcal{O}(N^{-1/2})$, whereas adaptive refinement leads to the optimal rate $\mathcal{O}(N^{-5/2})$.
The curves for the two estimators and the error are again parallel.
In Figure \ref{fig:slit ratio}, we plot 
the efficiency indices $\eta_h/\enorm{\phi-\phi_h}$ resp. $\mu_h/\enorm{\phi-\phi_h}$.
Figure~\ref{fig:slit errors} shows the errors of all considered adaptive IGABEM strategies.
Here, $\eta_h$-adaptive Galerkin IGABEM and $\mu_h$-adaptive collocation IGABEM lead to the best results.
In Figure \ref{fig:slit compare} we compare standard BEM  against IGABEM, where we use $\rho_h=\mu_h$.
While adaptve Galerkin IGABEM and adaptive standard BEM lead to optimal convergence rates, the best results are achieved with adaptive collocation IGABEM.


%% file: 04_aposteriori.tex

\subsection{Main results}\label{sec:main theorem}
For $T\in\TT_h$, we inductively define the patch $\omega_h^m(T)\subseteq \Gamma$ of order $m\in\N_0$ by
\begin{align}\label{eq:patch}
\begin{split}
 \omega_h^0(T) &:= T, \\
 \omega_h^{m+1}(T) &:= \bigcup\set{T'\in\TT_h}{T'\cap\omega_h^m(T)\neq\emptyset}
\end{split}
\end{align}
The main result of {Theorem \ref{thm:faermann}} requires the following two assumptions on $\TT_h$ and $\XX_h$
for some fixed integer $m\in\N_0$:
\begin{itemize}
\item[(A1)] For each $T\in\TT_h$, there exists some fixed function $\psi_T\in\XX_h$ with 
connected support $\supp(\psi_T)$ such that
\begin{align}\label{eq:phiT}
 T \subseteq \supp(\psi_T) \subseteq \omega_h^m(T).
\end{align}
\item[(A2)] There exists some constant $q\in (0,1]$ such that
\begin{align}\label{eq:contraction}
 \norm{1-\psi_T}{L^2(\supp(\psi_T))}^2 \le (1-q)\, |\supp(\psi_T)|
\end{align}
for all $T\in\TT_h$.
\end{itemize}
The first theorem shows that 
these assumptions are, in particular, satisfied for NURBS spaces.

\begin{theorem}[{\cite[Theorem 4.4]{fgp14}}]
\label{thm:NURBS satisfy Assumption} 
For $p\in \N_0$ and $m:=\lceil p/2\rceil$, the space   $\XX_h:=\widehat{\mathscr{N}}^p(\check{\mathcal{K}}_{h},\mathcal{W}_{h})$ 
satisfies the assumptions {\rm (A1)--(A2)}. The constant $0<q\le1$ depends only on $\kappa( \check{\mathcal{T}}_h)$, $\min(\mathcal{W}_h)$, $\max(\mathcal{W}_h)$, $p$, and $\gamma$.\qed
\end{theorem}

The main result of \cite{fgp14} reads as follows:

\begin{theorem}[{\cite[Theorem 3.1]{fgp14}}]
\label{thm:faermann}
For any approximation $\phi_h\in L^2(\Gamma)$, the residual $r_h=f-V\phi_h$ satisfies {the efficiency estimate}
\begin{align}\label{eq1:thm:faermann}
 \eta_h:=\Big(\sum_{{z}\in\NN_h}\eta_h(z)^2\Big)^{1/2}
 \le \Ceff\,\enorm{\phi-\phi_h}
\end{align}
with $\eta_h(z):=|r_h|_{H^{1/2}(\omega_h({z}))}$.
If the mesh $\TT_h$ and the discrete space $\XX_h$ satisfy assumptions~{\rm(A1)--(A2)}, the Galerkin solution $\phi_h\in\XX_h$ of~\eqref{eq:discrete} also satisfies the reliability estimate
\begin{align}\label{eq2:thm:faermann}
\enorm{\phi-\phi_h}
\le \Crel\,\eta_h.
\end{align}
The constant $\Ceff>0$ depends only on $\Gamma_{}$, while $\Crel>0$ additionally depends on $m$, $\kappa(\TT_h)$, and $q$.\qed
\end{theorem}

The following two theorems are the mathematical contributions of this work to the field of IGABEM.
They apply to both, Galerkin IGABEM as well as collocation IGABEM.

\begin{theorem}\label{thm:local}
For any approximation $\phi_h\in L^2(\Gamma)$ and $r_h:= f-V\phi_h$,
the indicator $\eta_h(z) := |r_h|_{H^{1/2}(\omega_h({z}))}$  
is bounded above by the weighted-residual indicator $\mu_h(z) :=|\omega_h(z)|^{1/2}\,\norm{r_h'}{L^2(\omega_h({z}))}$
\begin{align}\label{eq:local bound}
\eta_h(z) \le \sqrt{2}\,C_\Gamma\,\mu_h(z),
\end{align}
where $C_\Gamma>0$ is the constant from \eqref{eq:bi-Lipschitz}.
\end{theorem}

If collocation IGABEM as in Section \ref{subsec:collocation igabem} is used, the patch $\omega^{p+1}(T)$ contains a collocation point and therefore a root of the residual $r_h$, for each $T\in \TT_h$.
Hence, the condition of the following theorem is fulfilled with $m=p+1$.

\begin{theorem}\label{thm:reliable}
Suppose that wither $\phi_h\in\XX_h$ is the Galerkin solution of~\eqref{eq:discrete}, where $\XX_{h}$ satisfies {\rm(A1)--(A2)}, or that the residual $r_h=f-V\phi_h$ has at least one root in each $\omega^m_h(T)$ for all $T\in\TT_h$ and some fixed $m\in \N_0$. 
Then,
\begin{align}\label{eq:thm:residual}
&C_{\rm rel}^{-1}\,\enorm{\phi-\phi_h}
\le \mu_h 
:= \Big(\sum_{{z}\in\NN_h}\mu_h(z)\Big)^{1/2}
\end{align}
with $\mu_h(z):=|\omega_h(z)|^{1/2}\,\norm{r_h'}{L^2(\omega_h({z}))}$.
The constant $C_{\rm rel}>0$ depends only on $\Gamma$, $m$, $\kappa(\TT_h)$, and, in the first case,  $q$.
\end{theorem}

\subsection{Proof of Theorem~\ref{thm:local}}\quad
We only need the following lemma, whose proof is inspired by \cite[Proposition 2.2]{Nezza:2011yu}, where an analogous assertion for norms instead of seminorms is found.
The assertion itself is also stated in \cite[Lemma 7.4]{carsfaer} in a more general way.
Indeed a similar version of \eqref{eq:thm:bound} holds even for the $H^s$-seminorm, $0<s<1$. 
However, in \cite{carsfaer}, the proof  is only given for the hardest case $1/2<s<1$.

\begin{lemma}\label{lem:Hhseminorm leq H1seminorm}
For  any connected $\omega\subseteq \Gamma$, whose length satisfies $|\omega|\leq \frac{3}{4}L$ if $\Gamma=\partial \Omega$, there holds
\begin{align}\label{eq:thm:bound}
|u|_{H^{1/2}(\omega)}^2
 \le 2\,C_\Gamma^2\,|\omega|\,\norm{u'}{L^2(\omega)}^2
 \quad\text{for all }u\in H^1(\Gamma).
\end{align}
\end{lemma}

\begin{proof}
We recall that for a finite interval $I\subset \R$, $H^1(I)$ coincides with the space of all absolutely continuous functions on $\overline{I}$ with $L^2$ derivative; see, e.g., \cite[page 306]{evans}.

\noindent
{\bf Step 1:}
First we consider $I=(0,1)$ and prove 
\begin{equation}\label{eq:H-estimate 01}
|u|_{H^{1/2}(I)}^2\leq 2|{u}|_{H^1(I)}^2.
\end{equation}
We use the transformation theorem, with $r=\rho(s-t)+t$ and $s-t=\sigma$, as well as the Cauchy Schwarz inequality to get
\begin{align*}
|{u}|_{H^{1/2}(I)}^2&=\int_{I}\int_{I}{\left|\frac{u(s)-u(t)}{s-t}\right|^2}{\,ds}{\,dt}\\
&=\int_I\int_I{\left|\frac{\int_{(0,s)}{u'(r)}{\,dr}-\int_{(0,t)}{u'(r)}{\,dr}}{s-t}\right|^2}{\,ds}{\,dt}\\
&=\int_I\int_I{\left|\int_{I}{u'\big(\rho(s-t)+t\big)}{\,d\rho}\right|^2}{\,ds}{\,dt}\\
&\leq \int_I\int_I{\int_{I}{\big|u'\big(\rho(s-t)+t\big)\big|^2}{\,d\rho}}{\,ds}{\,dt}\\
&=\int_I\int_{(-t,1-t)}{\int_{I}{\big|u'\big(\rho\sigma+t\big)\big|^2}{\,d\rho}}{\,d\sigma}{\,dt}.
\end{align*}
We formally extend $u'$ by zero to $\R$.
This and the Fubini theorem lead to
\begin{align*}
|u|_{H^{1/2}(I)}^2&\leq\int_{I}\int_{(-1,1)}{\int_{I}{\big|u'\big(\rho\sigma+t\big)\big|^2}{\,d\rho}}{\,d\sigma}{\,dt}\\
&\leq \int_{I}\int_{(-1,1)}{\int_{\R}{\big|u'\big(\rho\sigma+t\big)\big|^2}{\,dt}}{\,d\sigma}{\,d\rho}\\
&= \int_{I}\int_{(-1,1)}{\norm{u'}{L^2(\R)}^2}{\,d\sigma}{\,d\rho}=2|u|_{H^1(I)}^2.
\end{align*}

\noindent
{\bf Step 2:}
If $I\subseteq \R$ is an arbitrary finite interval, it holds 
\begin{equation}\label{eq:H-estimate I}
|{u}|_{H^{1/2}(I)}^2\leq 2|I||{u}|_{H^1(I)}^2.
\end{equation}
Without loss of generality, let $I=(c,d)$ be open. 
We define the function $u_{(0,1)}:(0,1)\to \R:\tau\mapsto u\big(\tau(d-c)+c\big)$.
Obviously, it holds  $u_{(0,1)}\in H^1(0,1)$ with $u_{(0,1)}'(\tau)=(d-c)u'\big(\tau(d-c)+c\big)$. 
The transformation theorem with $s=\sigma(d-c)+c$, $t=\tau(d-c)+c$, and $r=\rho(d-c)+c$, and \eqref{eq:H-estimate 01} yield
\begin{align*}
&|{u}|_{H^{1/2}(I)}^2=\int_{I}\int_{I}{\left|\frac{u(s)-u(t)}{s-t}\right|^2}{\,ds}{\,dt}\\
&\quad=\int_{(0,1)}\int_{(0,1)}{\left|\frac{u\big(\sigma(d-c)+c\big)-u\big(\tau(d-c)+c\big)}{\sigma-\tau}\right|^2}{\,ds}{\,dt}\\
&\quad=|{u_{(0,1)}}|_{H^{1/2}(0,1)}^2\leq 2|{u_{(0,1)}}|_{H^{1/2}(0,1)}^2\\
&\quad=2\int_{(0,1)}{|u_{(0,1)}'(\rho)|^2}{\,d\rho}=2(d-c)\int_{I}{|u'(r)|^2}{\,dr}\\&\quad=2|I||{u}|_{H^1(I)}.
\end{align*}

\noindent
{\bf Step 3:}
We show \eqref{eq:thm:bound}.
 Let $I$ be a real interval with $\gamma_L(I)=\omega$. 
Then,  \eqref{eq:equivalent Hsnorm} and  \eqref{eq:H-estimate I} give
\begin{align*}
|u|_{H^{1/2}(\omega)}&=|{u}|_{H^{1/2}(\gamma_L(I))}^2\leq C_{\Gamma}^2|{u\circ\gamma_L}|_{H^{1/2}(I)}^2\\&\leq 2C_{\Gamma}^2|\omega| |{u\circ\gamma_L}|_{H^1(I)}^2=2C_\Gamma^2 |\omega|\norm{u'}{L^2(\omega)}^2.
\end{align*}
This concludes the proof.
\end{proof}
\subsection{Proof of Theorem~\ref{thm:reliable}}
We use the following estimate from \cite[Lemma~2.3]{faermann2d}; {see \cite[Proposition 2.13]{diplarbeit} for a detailed proof}.

\begin{lemma}\label{lem:Hnorm le Faer plus}
There exists a constant $C_1>0$ such that, for all $u\in H^{1/2}(\Gamma_{})$, it holds
\begin{align}
\norm{u}{H^{1/2}(\Gamma_{})}^2\leq \sum_{{z}\in \mathcal{N}_h} |u|_{H^{1/2}(\omega_h({z}))}^2+C_1\sum_{T\in\mathcal{T}_h} h_T^{-1}\norm{u}{L^2(T)}^2.
\end{align}
The constant only depends on $\Gamma_{}$ and $\kappa(\TT_h)$.\qed

\end{lemma}
\begin{proof}[Proof of Theorem~\ref{thm:reliable}]
If the residual is orthogonal to some $\XX_h$ satisfying {\rm(A1)--(A2)}, the assertion follows at once from  Theorem \ref{thm:faermann} in combination with Equation \eqref{eq:local bound}.
If the residual has local roots, we first note that
\begin{align}
\norm{\phi-\phi_h}{\H^{-1/2}(\Gamma)}\simeq \norm{f-V\phi_h}{H^{1/2}(\Gamma)}=\norm{r_h}{H^{1/2}(\Gamma)},
\end{align}
since $V$ is an isomorphism. The hidden constants only depend on $\Gamma$.

Taking $u=r_h$ in Lemma \ref{lem:Hnorm le Faer plus}, it only remains to estimate the sum $\sum_{T\in\mathcal{T}_h} h_T^{-1}\norm{r_h}{L^2(T)}^2$.
Note that shape regularity yields $|\omega_h^m(T)|\le (2m+1)\kappa(\TT_h)^{m}h_T$.
Replacing $T$ by $\omega^m(T)$, we apply Friedrich's inequality to see
\begin{align*}
&\sum_{T\in\mathcal{T}_h} h_T^{-1}\norm{r_h}{L^2(T)}^2\le\sum_{T\in\TT_h} h_T^{-1}\norm{r_h}{L^2(\omega_h^m(T))}^2\\
&\quad\le\sum_{T\in\TT_h} \frac{|\omega_h^m(T)|^2}{h_T}\norm{r_h'}{L^2(\omega^m(T))}^2\\
&\quad\le (2m+1)^2 \kappa(\TT_h)^{2m} \sum_{T\in\TT_h} h_T\norm{r_h'}{L^2(\omega_h^m(T))}^2\\
&\quad\le (2m+1)^3 \kappa(\TT_h)^{3m} \sum_{T\in\TT_h} h_T\norm{r_h'}{L^2(T)}^2\\
&\quad\le (2m+1)^3 \kappa(\TT_h)^{3m} \sum_{z\in\NN_h} |\omega_h(z)|\norm{r_h'}{L^2(\omega(z))}^2.
\end{align*}
This concludes the proof.
\end{proof}

%% file: 05_conclusion.tex

\subsection{Analytical results}

In this work, we considered adaptive BEM for weakly-singular integral equations $V\phi=f$ associated to elliptic PDEs in 2D. As model example served the 2D Laplacian, but the results apply as long as $V:\H^{-1/2}(\Gamma)\to H^{1/2}(\Gamma)$ is an elliptic isomorphism. With the residual $r_h:=f-V\phi_h$, we transferred the weighted-residual error estimator 
\begin{align}\label{conclusion:mu}
 \mu_h = \norm{h^{1/2}r_h'}{L^2(\Gamma)}
\end{align}
proposed in~\cite{cs96,cc97} from standard BEM with lowest-order polynomials to IGABEM, where we considered the Galerkin method as well as collocation. For either discretization, we proved that $\mu_h$ is reliable
\begin{align}\label{conclusion:reliable}
 \enorm{\phi-\phi_h} \le \Crel\,\mu_h;
\end{align}
see Theorem~\ref{thm:reliable}. In our preceding work~\cite{fgp14}, we considered the residual error estimator
\begin{align}\label{conclusion:eta}
 \eta_h = \Big(\sum_{z\in\NN_h}\int_{\omega_h({z})}\!\int_{\omega_h({z})}\hspace*{-4mm}\frac{|r_h(x)-r_h(y)|^2}{|x-y|^2}\,dy\,dx\Big)^{1/2}
\end{align}
proposed in~\cite{faermann2d}. In~\cite{fgp14}, we transferred this estimator from standard BEM with piecewise polynomials to IGABEM. Independently of the discretization, we  proved the general efficiency estimate 
\begin{align}\label{conclusion:efficient}
 \eta_h \le \Ceff\,\enorm{\phi-\phi_h},
\end{align}
while our proof of the converse estimate $\enorm{\phi-\phi_h} \le \Crel\,\eta_h$ is restricted to Galerkin IGABEM. However, the combination of~\eqref{conclusion:reliable} and~\eqref{conclusion:efficient} provides also full error control 
\begin{align}
 \Ceff^{-1}\,\eta_h \le \enorm{\phi-\phi_h} \le \Crel\,\mu_h
\end{align}
for collocation IGABEM computations in 2D. Moreover, this estimate implies the global relation $\eta_h \lesssim \mu_h$, and we even proved
\begin{align}
 \mu_h(z) \le C_{\rm loc}\, \eta_h(z)
 \quad\text{for all }z\in\NN_h
\end{align}
for the respective nodal contributions defined in~\eqref{eq:faermann} resp.\ \eqref{eq:residual};
see Theorem~\ref{thm:local} which holds independently of the discretization employed.

\subsection{Numerical results}

We proposed an adaptive algorithm which is capable to steer the mesh-refinement as well as the knot multiplicity in Galerkin and collocation IGABEM computations; see Algorithm~\ref{the algorithm}. Numerical experiments in Section~\ref{section:numerics} underline that generic singularities of the (unknown) exact solutions lead to reduced experimental convergence behavior if the underlying mesh is not appropriately graded. This is a well-known fact for standard BEM with piecewise polynomials, but also applies to IGABEM. 
Consequently, the gain of adaptive IGABEM (resp. the loss in case of uniform meshes) is huge due to the higher-order ansatz functions of IGABEM, and therefore adaptivity seems to be a must to exploit the full potential of isogeometric analysis.
In several numerical experiments, we showed that the proposed algorithm is capable to recover the optimal order of convergence. The gain of IGABEM is that the algorithm chooses smooth NURBS, where the exact solution appears to be smooth, while discontinuities and singularities are well detected and appropriately resolved. Compared to standard BEM with discontinuous piecewise polynomials, this leads to a smaller number of degrees of freedom for comparable accuracies.

For collocation IGABEM as well as Galerkin IGABEM and independently of the (uniform or adaptive) mesh-re\-fine\-ment, we observed that
\begin{align}
 \eta_h \simeq \enorm{\phi-\phi_h} \simeq \mu_h,
\end{align}
i.e., both error estimators are efficient and reliable. The efficiency indices $\eta_h/\enorm{\phi-\phi_h}$ and $\mu_h/\enorm{\phi-\phi_h}$ appeared to be $\le3$, i.e., the overestimation of the energy error is very moderate. 
  We note that only the equivalence $\eta_h \simeq \enorm{\phi-\phi_h}$ for Galerkin IGABEM as well as the bounds $\enorm{\phi-\phi_h}\lesssim \mu_h$ and $\eta_h\lesssim \enorm{\phi-\phi_h}$  have thoroughly been proved mathematically.

\subsection{Open questions and future work}
All considered numerical experiments show optimal convergence of the estimator and the error.
Understanding this observation mathematically in the spirit of \cite{axioms} is one of our goals for future research.
However, it is questionable if an analogous version of  the reduction property on refined element domains \cite[{\rm(A2)}]{axioms} can be proved for the Faermann estimator $\eta_h$.
Indeed, this is yet an open problem even for standard BEM with piecewise polynomials; see \cite{mitscha}, where at least convergence of an $h$-adaptive algorithm with $\eta_h$ is analyzed.
For the weighted-residual error estimator $\mu_h$ the axioms of \cite{axioms} are satisfied for standard Galerkin BEM with piecewise polynomials, see \cite[Section~5.4]{axioms}.  
For collocation IGABEM  
there remain two challenging mathematical questions:
First, one needs further investigation on the unique solvability of the discrete system.
Second, the quasi-orthogonality \cite[\rm(A3)]{axioms} is unclear for collocation methods.

As mentioned, we observed in all numerical experiments reliability as well as efficiency of the used error estimators.
However, it remains to mathematically verify the reliability estimate $\enorm{\phi-\phi_h}\lesssim \eta_h$ for collocation BEM and the efficiency estimate $\mu_h\lesssim \enorm{\phi-\phi_h}+{\rm osc}$,   at least for some higher-order oscillation terms ${\rm osc}$.
Again, these estimates are yet open problems even for standard BEM.
For lowest-order Galerkin BEM, the efficiency estimate is proved in \cite[Theorem 4]{feischl} under additional regularity assumptions on the Dirichlet data $g$ in \eqref{eq:Laplace}.

Finally, the ultimate goal is of course to analyze and apply  the estimators $\eta_h$ and $\mu_h$ in 3D Galerkin IGABEM.
For 3D one has to consider, e.g., T-splines \cite{igabem3d} or hierarchical B-splines \cite{buffa}, because, in contrast to multivariate NURBS, they naturally allow for  local mesh refinement.
\cite{faermann3d} shows that $\eta_h$ is reliable and efficient for standard BEM with piecewise polynomials, whereas \cite{cms} proves reliability for $\mu_h$.
In \cite[Section~5.4]{axioms} optimal convergence of adaptive $h$-refinement for $\mu_h$  is proved. The estimate $\eta_h\lesssim\mu_h$ as well as plain convergence for $\eta_h$-based adaptivity is analyzed in \cite{mitscha}.
The transfer of the mentioned results from standard BEM to adaptive IGABEM leaves interesting and challenging questions for future research.